\documentclass[12pt]{amsart}
\usepackage[latin1]{inputenc}	

\usepackage{soul}
\usepackage{hyperref}
\hypersetup{
    colorlinks=false,
    linkcolor=blue,
    filecolor=magenta,      
    urlcolor=cyan,
}

\usepackage[english]{babel}
\usepackage{amsmath,amsthm}
\usepackage{amssymb,latexsym}
\usepackage{graphicx}

\newcommand{\C}{\mathbb{C}}

\newcommand{\N}{\mathbb{N}}
\newcommand{\Q}{\mathbb{Q}}
\newcommand{\R}{\mathbb{R}}
\newcommand{\Z}{\mathbb{Z}}

\newtheorem{theorem}{Theorem}[section]

\theoremstyle{remark}
\newtheorem{remark}{Remark}[section]

\allowdisplaybreaks

\begin{document}
\author{Raffaele Marcovecchio}

\title[Vectors of Type II Hermite-Pad\'e approximations]{Vectors of 
				  Type II Hermite-Pad\'e approximations\\ and 
	   				a new linear independence criterion}

\dedicatory{Dedicated to Wadim Zudilin, with warm wishes, 
					on the occasion of his 50th birthday}
\address{}
\email{raffaele.marcovecchio@unich.it}

\date{July 4, 2020}

\begin{abstract}
We propose a linear independence criterion, and outline an application of it. 
Down to its simplest case, it aims at solving this problem: given three real numbers, 
typically as special values of analytic functions, how to prove that the $\Q$-vector 
space spanned by $1$ and those three numbers has dimension at least 3, whenever we are 
unable to achieve full linear independence, by using simultaneous approximations, 
i.e. those usually arising from Hermite-Pad\'e approximations of type II and their suitable 
generalizations. It should be recalled that approximations of type I and II are related, at 
least in principle: when the numerical application consists in specializing actual functional 
constructions of the two types, they can be obtained, one from the other, as explained in a 
well-known paper by K.Mahler \cite{Mahler}. That relation is reflected in a relation between 
the asymptotic behavior of the approximations at the infinite place of $\Q$. 
Rather interestingly, the two view-points split away regarding the asymptotic behaviors 
at finite places (i.e. primes) of $\Q$, and this makes the use of type II more convenient 
for particular purposes. In addition, sometimes we know type II approximations to a given set 
of functions, for which type I approximations are not known explicitly. Our approach can be 
regarded as a dual version of the standard linear independence criterion, which 
essentially goes back to Siegel. 
\end{abstract}


\maketitle

\section{Introduction} 
The systematically creative search of new irrationality (in the broad sense that includes 
linear independence) criteria is a flowering topic in number theory; see \cite{Bedulev}, 
\cite{BundschuhTopfer}, \cite{ChantasiriCriteria}, \cite{ChantasiriCriteres}, 
\cite{Colmez}, \cite{Dauguet}, \cite{FischlerOscillate}, \cite{FischlerVectors}, 
\cite{FischlerHussainKristensenLevesley}, \cite{FischlerZudilin}, \cite{LaurentRoy}, 
\cite{NesterenkoNumbers}, \cite{NesterenkoPadic}, \cite{Philippon}, \cite{TopferUber}, 
\cite{TopferAxI}, \cite{TopferAxII}, \cite{ZudilinDeterminant}, among many others. Their 
reviews on MathSciNet and Zentralblatt MATH provide an historically and methodologically 
well informed context. The development of this research topic is stimulated by a rich amount 
of problems, each of which carries intrinsic geometric aspects that require slightly, or 
sometimes substantially, new approaches. We believe that a novelty, whether substantial or 
just formal, is connected to suitable aspects of the established theory that are reflected in 
the way it should be communicated and interpreted. This introduction, as well as the many 
remarks we have placed all along in this paper, suggest a reception of the text according 
to the described tradition-oriented spirit, and it is no surprise that geometry, possibly in 
the modern algebra-focused meaning of the word, is again at the core of the investigation. We 
recommend the notes of course written by M.Waldschmidt \cite{Waldschmidt}, for a very useful 
introduction to the seminal work of Hermite, Lindemann, Weierstrass, Pad\'e, Gel'fond, 
Schneider, Siegel, Mahler,... on this topic.
\subsection{Background}	
Suppose we have $\gamma_1,\dots,\gamma_m$ real numbers, and we wish to prove that the 
$m+1$ numbers $1,\gamma_1,\dots,\gamma_m$ are linearly independent over the field $\Q$ 
of all rational numbers. One efficient method, when successful, is to construct $m$ 
sequences of small linear forms satisfying the following properties ($\mu=1,\dots,m$):
\begin{itemize}
\item[{\it i})] $\varepsilon_n^{(\mu)} = q_n  \gamma_\mu - p_n^{(\mu)}\in \Z\gamma_\mu+\Z$;

\item[{\it ii})] $\varepsilon_n^{(\mu)}\to 0$ for $n\to\infty$;

\item[{\it iii})] for all $\boldsymbol{\lambda}\in\Z^m\setminus\{\mathbf{0}\}$ 
the inner product $\boldsymbol{\lambda} \cdot \boldsymbol{\varepsilon}_n$ is non-zero 
for infinitely many $n$. 
\end{itemize}
In a reasonably convenient situation, the assumption {\it ii}) is achieved by standard, 
although possibly difficult to apply, analytic methods, while the requirement in {\it iii}) 
looks elusive, and is indeed, at least sometimes, subtle to handle. 

More precisely, those analytic methods are applied to suitable linear 
forms $\widetilde{\varepsilon}_n^{\, (\mu)}=\widetilde{q}_n \gamma_\mu 
									- \widetilde{p}_n^{\, (\mu)}$ with 
{\it rational} coefficients, in such a way 
that $\varepsilon_n^{(\mu)} = D_n \widetilde{\varepsilon}_n^{\, (\mu)}$ and that 
the asymptotic behavior of the sequence of integers $D_n$ is controlled through a 
careful application of the Prime Number Theorem, using refined tools of different kinds, 
pioneered in papers like \cite{Rukhadze} and \cite{Dubickas}, on one hand, and 
\cite{RhinViola1996}, on the other hand; see \cite{Viola} for an historical perspective, 
and \cite{Marcovecchio2016} for a tentative unification of different approaches.

In some cases, one can replace {\it iii)} by both of
\begin{itemize}
\item[{\it iii.a)}] $1,\gamma_1,\dots,\gamma_{m-1}$ are linearly independent over $\Q$;

\item[{\it iii.b)}] for all $\boldsymbol{\lambda}\in\Z^{m-1}$, 
				$|\lambda_1 \varepsilon_n^{(1)} + \cdots 
				+ \lambda_{m-1} \varepsilon_n^{(m-1)}|< \varepsilon_n^{(m)}$ 
				for infinitely many $n$.
\end{itemize}
This is the strategy followed in 
\cite[Lemma 5.1]{ViolaZudilin} and \cite[(4-6)-(4-7)]{RhinViola2019}. The methods in 
\cite{RhinViola2019} can be, in principle, extended to the situation in 
\cite{Marcovecchio2016}, by playing the `wild card': the $\C^N$-saddle point method 
nicely developed in \cite{PinnaViola}.
\subsection{The new ingredient}
Let us switch to those awkward situations where {\it ii}) does not hold, 
in spite of our effort to arrange such a crucial requirement. Accordingly, 
we confine our goal to merely proving that
\begin{equation}		\label{Eq_Sec1.2}
\dim_\Q (\Q+\Q\gamma_1+\cdots+\Q\gamma_m) \geqslant m,
\end{equation}
say, which, at least, is less demanding. If, for instance, $m=3$, 
$\gamma_1={\rm Li}_1(\frac{1}{z})=-{\rm Li}_1(-\frac{1}{1-z})$, 
$\gamma_2={\rm Li}_2(\frac{1}{z})$, $\gamma_3={\rm Li}_2(-\frac{1}{1-z})$, then 
$2(\gamma_2+\gamma_3)=-\gamma_1^2$, and $\gamma_1$ is transcendental for all 
algebraic $z\not=0,1$. In several interesting cases \eqref{Eq_Sec1.2} is an open problem: 
the numbers $\gamma_i$ can be chosen to be values of polylogarithms at several 
points \cite{DavidHirata-KohnoKawashima}, or multiple polylogarithms \cite{Goncharov} 
\cite{FischlerRivoal_multiple}; or your favorite to-be-proved-linearly-independent `constants'. 

The key idea is to deal with \eqref{Eq_Sec1.2} when there is no available analytic 
construction of functional approximations to $m-1$ functions, that is to say: we are not 
able to identify a subspace of dimension $m$. We are about to suggest that a strategy 
to achieve \eqref{Eq_Sec1.2}, maybe for some other numbers unrelated to specific examples 
above, can still be quite well akin to {\it i})-{\it iii}). We should indeed construct 
(virtually: see 1.3 below) $2m$ sequences of linear forms such that: 
\begin{enumerate}
\item[{\it iv})] $\varepsilon_n^{(\mu,\nu)} = q_n^{(\nu)}\gamma_\mu - p_n^{(\mu,\nu)} \in 
					\Z\gamma_\mu+\Z$\quad ($\mu=1,\dots,m;v=1,2$); 

\item[{\it v})] $\Delta_n^{(\mu_1,\mu_2)} 
		= \varepsilon_n^{(\mu_1,1)} \varepsilon_n^{(\mu_2,2)} 
		- \varepsilon_n^{(\mu_1,2)} \varepsilon_n^{(\mu_2,1)} \to 0$ 
		for $n\to\infty$ ($\mu_1,\mu_2=1,\dots,m$);
		
\item[{\it vi})] for all $\boldsymbol{\lambda}\in\mathcal{M}(2,m;\Z)$ with rank $2$,  
the matrix $\boldsymbol{\lambda} \cdot \boldsymbol{\varepsilon}_n\in\mathcal{M}(2,2;\Z)$ 
is non-singular for infinitely many $n$.
\end{enumerate}
If suitably related sequences $\widetilde{\varepsilon}_n^{\, (\mu,\nu)}$ 
(see 2.5 below for more precision) are holonomic, in the sense of \cite{WilfZeilberger}, 
then the sequences of determinants in {\it v}) are holonomic as well, so that the linear 
recurrence equation satisfied by the determinants can be computed from the equation 
satisfied by their entries, and {\it v}) can be achieved through the 
Poincare-Perron-Pituk's theorems. Moreover, 
\begin{align*}
\Delta_n^{(\mu_1,\mu_2)}	 & = 
(q_n^{(1)}\gamma_{\mu_1} - p_n^{(\mu_1,1)})(q_n^{(2)}\gamma_{\mu_2} - p_n^{(\mu_2,2)}) -
(q_n^{(1)}\gamma_{\mu_2} - p_n^{(\mu_2,1)})(q_n^{(2)}\gamma_{\mu_1} - p_n^{(\mu_1,2)}) \\ 
& = (p_n^{(\mu_1,1)} p_n^{(\mu_2,2)} - p_n^{(\mu_2,1)} p_n^{(\mu_1,2)}) 
  + (p_n^{(\mu_2,1)} q_n^{(2)} - p_n^{(\mu_2,2)} q_n^{(1)}) \gamma_{\mu_1} 
  + (p_n^{(\mu_1,2)} q_n^{(1)} - p_n^{(\mu_1,1)} q_n^{(2)}) \gamma_{\mu_2}
\end{align*}
is just a linear form in $1,\gamma_{\mu_1},\gamma_{\mu_2}$; for this reason, the case 
$m=2$ simply resolves into an application of Hermite-Pad\'e approximations of type I. 
The simplest new case is therefore $m=3$. In the general situation below, we 
replace the lower bound $m$ for the dimension with the lower bound $m+2-l$ in 
\eqref{Eq_Sec1.2}, and consider $lm$ linear forms, instead of $m$ as 
in {\it i})-{\it iii}), or $2m$ as in {\it iv})-{\it vi}). We handle 
the main theorems in section 2,  with some variations on the same theme 
and a refinement on the main result that should incorporate most of potential 
applications.
\subsection{Auxiliary results between modern and neoclassical}
Section 3 is devoted to an auxiliary tool: in many concrete situations, the 
sequences $q_n,p_n^{(\mu)},\varepsilon_n^{(\mu)}$ in {\it i})-{\it iii}) satisfy 
the same Poincar\'e-Perron-Pituk-type linear recurrence equation of order $m+1$, see 
\cite[(2.1), (4.4) and (4.6)]{Marcovecchio2014} for examples with $m$ arbitrary. In 
such a special situation it is convenient to simply take
\[
q_n^{(\nu)}=q_{n+\nu-1}, \quad p_n^{(\mu,\nu)}=p_{n+\nu-1}^{(\mu)}
\]
in {\it iv})-{\it vi}), where $q_n$ and $p_n^{(\mu)}$ are the sequences 
in {\it i})-{\it iii}). To achieve our program, one possible strategy is to find 
a recurrence equation satisfied by the minors involved. In the special case $m=3$ 
and $l=2$, this new recurrence equation has order $3$, and, roughly, is (up to a suitable 
normalization) the adjoint of the recurrence equation for $\varepsilon_n^{(\mu)}$. 
For general $1\leqslant l\leqslant m$, the linear recurrence equation for the 
determinants $\Delta_n^{(\mu_1,\dots,\mu_l)}$ similar to $\Delta_n^{(\mu_1,\mu_2)}$ 
in {\it vii}) has order $\binom{m}{l}$. However, as we shall show, there is a better 
way to achieve the main goal of this program, which is to find a technical tool that 
helps us to deal with {\it v}) and {\it vi}) above, and their generalizations to lower 
dimensions $m+2-l$ in \eqref{Eq_Sec1.2}. We find very precisely the asymptotic behavior 
of the minors $\Delta_n^{(\mu_1,\dots,\mu_l)}$ in quite a general setting, which should 
suffice for applications. Most of the material in this section is well-known. We 
apply a Golden Oldie, the Sylvester-Franke Theorem, to the compound matrix of the 
Casoratian matrix of a difference equation, and combine that with a general result 
by Pituk on the asymptotic behavior of solutions of a Poincar\'e-Perron-type difference 
system. 

A completely different problem is to {\it find} a recurrence equation satisfied by 
$q_n,p_n^{(\mu)},\varepsilon_n^{(\mu)}$ (of course, if it exists in the first place), 
and how to determine the roots of the characteristic polynomial of that equation. This 
can be achieved in different ways, see e.g. \cite{Marcovecchio2014} and \cite{ZudilinZeta2}.

There is a small inconsistency between the notations in Sect. 2 and Sect. 3. When we 
consider a recurrence of order $m$ in Sect. 3, we aim at applications to the linear 
independence of $m$ numbers $1,\gamma_1,\dots,\gamma_{m-1}$ (not of $m+1$ numbers), 
in the sense of a lower bound for the dimension, as explained above. We hope that this 
will not trouble the reader, and we anticipate that the expositions in the two sections 
are quit autonomous from each other.
\subsection{Two examples to begin with}
In section 4 we sketch a possible application of our criterion. We give two examples to 
illustrate the results from Sect.2 and from Sect.3 separately. We do so building, in both 
cases, on the clever construction introduced in \cite{DavidHirata-KohnoKawashima}. 
In principle, the results from Sect.2 and Sect.3 could be combined together in the 
second example, but the implementation seems to be difficult. 

The simplest example we can produce involves five numbers, 
namely $1,{\rm Li}_1(\frac{1}{q}),{\rm Li}_2(\frac{1}{q}),
	{\rm Li}_1(\frac{2}{q})$, ${\rm Li}_2(\frac{2}{q})$. 
It is still open to find an example involving only four numbers. Our second example is with 
the $km+1$ numbers $1,{\rm Li}_1(\frac{1}{q}),\dots,{\rm Li}_k(\frac{1}{q}),\dots,
	{\rm Li}_1(\frac{1}{mq}),\dots,{\rm Li}_k(\frac{1}{mq})$.

We hope that further, and perhaps more interesting, applications will be provided 
in some future. 
\section{Main results}
Throughout this paper, we use the following abbreviations 
for a matrix with $u$ rows and $v$ columns:
\[
\boldsymbol{\rho} = \Big[ \rho^{(i,j)} \Big]_{\substack{i=1,\dots,u \\ j=1,\dots,v}} = 
\left[
\begin{matrix}	
\rho^{(1,1)} & \cdots & \rho^{(1,v)} \\
\vdots		& \ddots & \vdots \\
\rho^{(u,1)} & \cdots  & \rho^{(u,v)} 
\end{matrix}
\right].
\]
We often deal with a sub-matrix of $\boldsymbol{\rho}$:
\[
\boldsymbol{\rho}^{(\boldsymbol{\mu},\boldsymbol{\nu})} = 
\Big[ \rho^{(\mu_i,\nu_j)} \Big]_{\substack{i=1,\dots,k \\ j=1,\dots,l}} = 
\left[
\begin{matrix}	
\rho^{(\mu_1,\nu_1)} & \cdots & \rho^{(\mu_1,\nu_l)} \\
\vdots		& \ddots & \vdots \\
\rho^{(\mu_k,\nu_1)} & \cdots  & \rho^{(\mu_k,\nu_l)} 
\end{matrix}
\right].
\]
Sometimes we replace $\boldsymbol{\mu}$ (resp.: $\boldsymbol{\nu}$) with $k$ 
(resp.: $l$) when we select the first $k$ rows (resp.: the first $l$ columns). 
A sub-matrix of $\boldsymbol{\rho}$ with $u$ rows (resp.: with $v$ columns) 
is denoted by $\boldsymbol{\rho}^{(\textbf{-},\boldsymbol{\nu})}$ (resp.: 
by $\boldsymbol{\rho}^{(\boldsymbol{\mu},\textbf{-})}$), while the dash is used, 
after the square brackets and in place of the range for the rows (or the columns), when 
the concerned matrix has one row (or one column). We write $\boldsymbol{I}^{(u,u)}$ for 
the identity matrix, and $\boldsymbol{0}^{(u,v)}$ for the zero matrix. We denote 
by ${}^t\! \boldsymbol{\rho}$ the transpose of $\boldsymbol{\rho}$. We also 
consider matrices divided by blocks, for 
example $\boldsymbol\rho = [\boldsymbol\rho^{(\textbf{-},(1,\dots,k))}\;|\; 
\boldsymbol\rho^{(\textbf{-},(k+1,\dots,v))}]$.
\subsection{The main criterion}
Let $\gamma_1,\dots,\gamma_m$ be real numbers, and let $l\in\{1,\dots,m\}$. 
Let also $q_n^{(\nu)},p_n^{(1,\nu)},\dots,p_n^{(m,\nu)}$, with $\nu=1,\dots,l$, 
be $l(m+1)$ sequences of integers (hereafter: elements of $\Z$), and put
\begin{equation}		\label{Eq_Sec2.1qgammap}
\varepsilon_n^{(\mu,\nu)}:= q_n^{(\nu)} \gamma_\mu - p_n^{(\mu,\nu)} \quad 
														(\mu=1,\dots,m;\nu=1,\dots,l).
\end{equation}
\begin{theorem}
Suppose that for all choices of $l$ (distinct) 
indices $\boldsymbol{\mu}=(\mu_1,\dots,\mu_l)$ 
from $1$ to $m$,
\begin{equation}		\label{Eq_Sec2.1detto0}
\det \boldsymbol{\varepsilon}_n^{(\boldsymbol{\mu},\textbf{-} )} = 
\det \Big[\varepsilon_n^{(\mu,\nu)}\Big]_{\substack{\mu=\mu_1,\dots,\mu_l \\ 
											\nu=1,\dots,l} } \to 0	\quad (n\to\infty).
\end{equation}
Furthermore, suppose that for all $\lambda^{(i,j)}\in\Z$ ($i=1,\dots,l;j=1,\dots,m$) 
such that the matrix
\[
\boldsymbol{\lambda} = 
\Big[\lambda^{(i,j)}\Big]_{\substack{i=1,\dots,l \\ j=1,\dots,m}}
\]
has rank $l$, the square matrix
\[
\boldsymbol{\lambda} \boldsymbol{\varepsilon}_n =  
\Big[\;\cdot\; {}_n^{(i,\nu)}\Big]_{\substack{i=1,\dots,l \\ \nu=1,\dots,l}} = 
\Big[\lambda^{(i,j)}\Big]_{\substack{i=1,\dots,l \\ j=1,\dots,m}} 
\Big[\varepsilon_n^{(\mu,\nu)}\Big]_{\substack{\mu=1,\dots,m \\ \nu=1,\dots,l} }
\]
is non-singular for infinitely many $n$. 

Then
\begin{equation}		\label{Eq_Sec2.1dimQ}
\dim_\Q ( \Q + \Q\gamma_1 + \cdots + \Q\gamma_m) \geqslant 2+m-l.
\end{equation}
\end{theorem} 
\begin{proof}
Arguing by contradiction, thus allowing 
\[
\dim_\Q ( \Q + \Q\gamma_1 + \cdots + \Q\gamma_m) \leqslant 1+m-l,
\]
let $\varpi^{(i)},\omega^{(i,j)}\in\Z$ be such that 
\[
\varpi^{(i)}+\omega^{(i,1)}\gamma_1+\cdots+\omega^{(i,m)}\gamma_m = 0 \quad (i=1,\dots,l),
\]
i.e. $\boldsymbol{\varpi} + \boldsymbol{\omega} \boldsymbol{\gamma}= {\bf 0}$,
and that $[\boldsymbol{\varpi} | \boldsymbol{\omega}]$ and $\boldsymbol{\omega}$ 
have rank $l$. We shorten \eqref{Eq_Sec2.1qgammap} by $\boldsymbol{\varepsilon}_n = 
\boldsymbol{\gamma} \boldsymbol{q}_n - \boldsymbol{p}_n$, thus obtaining 
\[
\boldsymbol{\omega} \boldsymbol{\varepsilon}_n = 
\boldsymbol{\omega}\boldsymbol{\gamma} \boldsymbol{q}_n 
	- \boldsymbol{\omega} \boldsymbol{p}_n = 
	- (\boldsymbol{\varpi} \boldsymbol{q}_n + \boldsymbol{\omega} \boldsymbol{p}_n) 
	= - \left[ \boldsymbol{\varpi}\; |\; \boldsymbol{\omega} \right]\ 
	\left[ \begin{matrix} \boldsymbol{q}_n \\ - \\ \boldsymbol{p}_n \end{matrix} \right].
\] 
By our assumptions, this must be a non-singular matrix, so that its determinant 
has to be a non-zero integer, for infinitely many $n$. On the other hand, by the  
Binet-Cauchy formula for the determinant of a product of two matrices,
\[
\det \boldsymbol{\omega} \boldsymbol{\varepsilon}_n = 
\sum_{\boldsymbol{\mu}} \det \boldsymbol{\omega}^{( \textbf{-}, \boldsymbol{\mu})}
						\det \boldsymbol{\varepsilon}_n^{(\boldsymbol{\mu},\textbf{-} )},
\] 
where the sum is over all multi-indices ${\boldsymbol\mu}$ such that $1\leqslant
\mu_1<\cdots<\mu_l\leqslant m$. Therefore, using \eqref{Eq_Sec2.1detto0},
$ \det \boldsymbol{\omega} \boldsymbol{\varepsilon}_n \to 0$ as $n\to\infty$. This 
contradiction ends the proof of \eqref{Eq_Sec2.1dimQ}.
\end{proof}
\begin{remark}
By the Binet-Cauchy formula, the assumption \eqref{Eq_Sec2.1detto0} is equivalent to
\[
\det {}^t \boldsymbol{\varepsilon}_n \boldsymbol{\varepsilon}_n 
= \sum_{\boldsymbol{\mu}} 
		\det\! {}^2\, \boldsymbol{\varepsilon}_n^{(\boldsymbol{\mu},\textbf{-} )} 
		\to 0 \quad (n\to\infty),
\]
with an interesting interpretation of the determinant as the square of the
$l$-dimensional volume of the parallelotope generated by the $l$ columns of 
$\boldsymbol{\varepsilon}_n$ in $\R^m$. 

Also, the validity of the non-vanishing assumption is checked more easily 
if we have a prior partial information on the linear independence of some 
numbers among $1,\gamma_1,\dots,\gamma_m$, in analogy to the situation outlined 
in the introduction.
\end{remark}
\begin{remark}
It's worth noticing that each $\det \boldsymbol{\varepsilon}_n^{(\boldsymbol{\mu},
\textbf{-} )}$ is linear combination of $1,\gamma_{\mu_1},\dots,\gamma_{\mu_l}$, because
\[
\det \boldsymbol{\varepsilon}_n^{(\boldsymbol{\mu},\textbf{-} )} = 
\det \left[\boldsymbol{\gamma q}_n - \boldsymbol p_n\right]^{(\boldsymbol{\mu},\textbf{-} )}
= \det  
\left[ \begin{matrix}  
	\boldsymbol{\gamma} & | & \boldsymbol{p}_n \\ - &  & - \\ 1 & | & \boldsymbol{q}_n  						\end{matrix} \right]^{(\boldsymbol{\widehat{\mu}},\textbf{-} )},
\]
where $\boldsymbol{\widehat{\mu}}=(\mu_1,\dots,\mu_l,l+1)$, as is easily seen from
\[
\left[ \begin{matrix}  \boldsymbol{\gamma q}_n - \boldsymbol{p}_n & | 
				& \boldsymbol{\gamma} \\ - &  & - \\ 
					\boldsymbol{0}^{(\textbf{-},l)} & | &  1	\end{matrix} \right] =
\left[ \begin{matrix}  \boldsymbol{\gamma} & | & \boldsymbol{p}_n \\ - &  & - \\ 
					1 & | & \boldsymbol{q}_n  	\end{matrix} \right]
\left[ \begin{matrix} \boldsymbol{q}_n & | & 1 \\ - &  & - \\ 
				 - \boldsymbol{I}^{(l,l)} & | & \boldsymbol{0}^{(l,\textbf{-})}  \end{matrix} \right].
\] 
This generalizes an observation we made in Sect.~1.2 for the case $l=2$.
\end{remark}
\subsection{First variation}
By repeating the same proof as in Theorem 2.1, we have
\begin{theorem}
Suppose that for all choices of $l$ (distinct) 
indices $\boldsymbol{\mu}=(\mu_1,\dots,\mu_l)$ 
from $1$ to $m$,
\[
\det \boldsymbol{\varepsilon}_n^{(\boldsymbol{\mu}, \textbf{-} )} = 
\det \Big[\varepsilon_n^{(\mu,\nu)}\Big]_{\substack{\mu=\mu_1,\dots,\mu_l \\ 
											\nu=1,\dots,l} } \to 0	\quad (n\to\infty).
\]
Furthermore, suppose that 
for all $\theta^{(i)},\lambda^{(i,j)}\in\Z$ ($i=1,\dots,l;j=1,\dots,m$) 
such that the matrix
\[
\left[\boldsymbol{\theta}\; |\; \boldsymbol{\lambda} \right] = 
\left[ \left[ \theta^{(i)}\right]_{\substack{i=1,\dots,l \\ - }} \ | \ 
\left[ \lambda^{(i,j)}\right]_{\substack{i=1,\dots,l \\ j=1,\dots,m}} \right]
\]
has rank $l$, the square matrix
\begin{align*}
[\boldsymbol{\theta}\; |\; \boldsymbol{\lambda} ] 
\left[ \begin{matrix} \boldsymbol{q}_n \\ - \\ \boldsymbol{p}_n \end{matrix} \right] & =
\boldsymbol{\theta} \boldsymbol{q}_n + \boldsymbol{\lambda} \boldsymbol{p}_n \\ & = 
\Big[\theta^{(i)}\Big]_{\substack{i=1,\dots,l \\ - }} 
\Big[q_n^{(\nu)}\Big]_{\substack{ - \\ \nu=1,\dots,l} } 
+
\Big[\lambda^{(i,j)}\Big]_{\substack{i=1,\dots,l \\ j=1,\dots,m }} 
\Big[p_n^{(\mu,\nu)}\Big]_{\substack{ \mu=1,\dots,m \\ \nu=1,\dots,l} }  
\end{align*}
is non-singular for infinitely many $n$. 

Then
\[
\dim_\Q ( \Q + \Q\gamma_1 + \cdots + \Q\gamma_m) \geqslant 2+m-l. 
\]
\end{theorem} 
\subsection{Second variation} 
It may be appropriate to have alternative versions of the above criterion, for use in 
different situations. Therefore, let us change our setting a little bit. We still 
have $m$ real numbers $\gamma_1,\dots,\gamma_m$ and $1\leqslant l\leqslant m$. Now, 
let $q_n^{(\nu)};p_n^{(1,\nu)},\dots,p_n^{(m,\nu)}$ ($\nu=0,\dots,m$) be $(m+1)^2$  
sequences in $\Z$. Let us extend the notation in \eqref{Eq_Sec2.1qgammap} accordingly:
\[
\varepsilon_n^{(\mu,\nu)}:= q_n^{(\nu)} \gamma_\mu - p_n^{(\mu,\nu)} \quad 
														(\mu=1,\dots,m;\nu=0,\dots,m).
\]
\begin{theorem}
Suppose that for all choices of (distinct) $l$ 
indices $\boldsymbol{\mu}=(\mu_1,\dots,\mu_l)$ from $1$ to $m$ 
and $\boldsymbol{\nu}=(\nu_1,\dots,\nu_l)$ from $0$ to $m$,
\[
\det \boldsymbol{\varepsilon}_n^{(\boldsymbol{\mu},\boldsymbol{\nu})}  
		\to 0	\quad (n\to\infty).
\]
Furthermore, suppose that 
\[
\det \left[ \begin{matrix} \boldsymbol{q}_n \\ - \\ \boldsymbol{p}_n \end{matrix} \right] 
 \not = 0 \quad (n=0,1,2,\dots) .
\]
Then
\[
\dim_\Q ( \Q + \Q\gamma_1 + \cdots + \Q\gamma_m) \geqslant 2+m-l.
\]
\end{theorem}
\begin{proof}
As above, we argue by contradiction, and, like in the previous proof, we have
\[
\boldsymbol{\omega} \boldsymbol{\varepsilon}_n = 
\boldsymbol{\omega}\boldsymbol{\gamma} \boldsymbol{q}_n 
	- \boldsymbol{\omega} \boldsymbol{p}_n = 
	- (\boldsymbol{\varpi} \boldsymbol{q}_n + \boldsymbol{\omega} \boldsymbol{p}_n) = 
	- [ \boldsymbol{\varpi}\; |\; \boldsymbol{\omega} ]\ 
	\left[ \begin{matrix} \boldsymbol{q}_n \\ - \\ \boldsymbol{p}_n \end{matrix} \right],
\]
but now $\boldsymbol{\omega} \boldsymbol{\varepsilon}_n$ is a matrix with $l$ rows 
and $m+1$ columns. Since $[ \boldsymbol{\varpi} | \boldsymbol{\omega} ]$ has rank $l$ 
and $[ {}^t \boldsymbol{q}_n | {}^t \boldsymbol{p}_n ]$ is non-singular, 
their product, which is $\boldsymbol{\omega} \boldsymbol{\varepsilon}_n$, has rank $l$. 
By the pigeonhole principle, there exists a square sub-matrix $\boldsymbol{\omega} 
			\boldsymbol{\varepsilon}_n^{(\textbf{-}, \boldsymbol{\nu})}$, 
with $\boldsymbol{\nu}=(\nu_1,\dots,\nu_l)$ and $0\leq \nu_1<\cdots<\nu_l\leq m$, 
that is non-singular for infinitely many $n$. For that sub-matrix we have
$\boldsymbol{\omega} \boldsymbol{\varepsilon}_n^{( \textbf{-},\boldsymbol{\nu})} = 
		- (\boldsymbol{\varpi} \boldsymbol{q}_n^{( \textbf{-}, \boldsymbol{\nu})} 
		+ \boldsymbol{\omega} \boldsymbol{p}_n^{( \textbf{-}, \boldsymbol{\nu})})$,
so that the determinant must be a non-zero integer for infinitely many $n$. 
For concluding the proof we apply the Binet-Cauchy formula, to obtain
\[
\det \boldsymbol{\omega} \boldsymbol{\varepsilon}_n^{(\textbf{-},\boldsymbol{\nu})} = 
\sum_{\boldsymbol{\mu}} 
			\det \boldsymbol{\omega}^{(\textbf{-}, \boldsymbol{\mu})}
			\det \boldsymbol{\varepsilon}_n^{(\boldsymbol{\mu}, \boldsymbol{\nu}) } 
			\to 0 \quad (n\to\infty). \qedhere
\] 
\end{proof}
\begin{remark}
The last theorem is designed to be used when $q_n^{(\nu)},p_n^{(\mu,\nu)}$ are 
specializations of a system of type II Hermite-Pad\'e approximations to $m$ 
functions. In this case, indeed, the non-vanishing of the determinant follows, 
more or less routinely, by analytic properties that characterize the polynomials 
involved. 
\end{remark}
\subsection{A refinement}
The above criteria can be refined following an idea I learned from F.Amoroso \cite{Amoroso},   
see also \cite{Colmez} and Remark 2.6 below. A trickier use of this idea lead in 
\cite{FischlerZudilin} to a refinement of Nesterenko's criterion with very interesting 
applications to the linear independence of zeta values and related numbers. As in 
\eqref{Eq_Sec2.1qgammap}, let 
\[
\varepsilon_n^{(\mu,\nu)}:= q_n^{(\nu)} \gamma_\mu - p_n^{(\mu,\nu)} \quad 
														(\mu=1,\dots,m;\nu=1,\dots,l),
\]
but here $q_n^{(\nu)}, p_n^{(\mu,\nu)}$ are rational numbers. 
Let $D_n^{(1)},\dots,D_n^{(m)},\delta_n^{(1)},\dots,\delta_n^{(l)}$ be 
positive integers, and suppose that
\begin{equation}			\label{Eq_Sec2.4refinement}
\frac{q_n^{(\nu)}}{\delta_n^{(\nu)}}\in\Z, \; 
\frac{D_n^{(\mu_2)}}{\delta_n^{(\nu)}}\, p_n^{(\mu_1,\nu)}\in\Z \quad 
							(1\leqslant\mu_1\leqslant\mu_2\leqslant m;\nu=1,\dots,l).
\end{equation}
We have the following
\begin{theorem}[Refinement of Theorem 2.1]
Besides the assumptions above, suppose that for all choices of $l$ (distinct) 
indices $\boldsymbol{\mu}=(\mu_1,\dots,\mu_l)$ from $1$ to $m$,
\[
\frac{D_n^{(m)}}{\delta_n^{(1)}}\cdots\frac{D_n^{(m-l+1)}}{\delta_n^{(l)}}\,
\det \boldsymbol{\varepsilon}_n^{(\boldsymbol{\mu},\textbf{-} )} \to 0 \quad (n\to\infty).
\]
Furthermore, suppose that for all $\lambda^{(i,j)}\in\Z$ ($i=1,\dots,l;j=1,\dots,m$) 
such that the matrix
\[
\boldsymbol{\lambda} = 
\Big[\lambda^{(i,j)}\Big]_{\substack{i=1,\dots,l \\ j=1,\dots,m}}
\]
has rank $l$, the square matrix $\boldsymbol{\lambda \varepsilon}_n$
is non-singular for infinitely many $n$. 

Then
\[
\dim_\Q ( \Q + \Q\gamma_1 + \cdots + \Q\gamma_m) \geqslant 2+m-l.
\]
\end{theorem} 
\begin{proof}
We may argue as in the proof of Theorem 2.1 above, the only difference being that, 
with the notations therein for $\boldsymbol\omega$ and $\boldsymbol\varpi$, here 
\[
\frac{D_n^{(m)}}{\delta_n^{(1)}}\cdots\frac{D_n^{(m-l+1)}}{\delta_n^{(l)}} 
\det \boldsymbol{\omega \varepsilon}_n
\]
is a non-zero integer, because
\[
\det \boldsymbol{\omega \varepsilon}_n 
	= - \det \left[ \boldsymbol{\varpi}\; |\; \boldsymbol{\omega} \right]\ 
	\left[ \begin{matrix} \boldsymbol{q}_n \\ - \\ \boldsymbol{p}_n \end{matrix} \right] 
	= - \sum_{\boldsymbol{\xi}} 
		\det \left[ \boldsymbol{\varpi}\; |\; \boldsymbol{\omega} 
												\right]^{(\textbf{-}, \boldsymbol{\xi})}
		\det \left[ \begin{matrix} \boldsymbol{q}_n \\ - \\ \boldsymbol{p}_n \end{matrix} 														\right]^{(\boldsymbol{\xi}, \textbf{-})},
\]
and each
\[
\frac{D_n^{(m)}}{\delta_n^{(1)}}\cdots\frac{D_n^{(m-l+1)}}{\delta_n^{(l)}} 
\det \left[ \begin{matrix} \boldsymbol{q}_n \\ - \\ \boldsymbol{p}_n \end{matrix} 														\right]^{(\boldsymbol{\xi}, \textbf{-})}, 
			\qquad 1\leqslant \xi_1<\cdots<\xi_l\leqslant m+1,
\]
is an integer, as is easily seen on multiplying the $i$-th row of 
\[
\left[ \begin{matrix} \boldsymbol{q}_n \\ - \\ \boldsymbol{p}_n \end{matrix} 														\right]^{(\boldsymbol{\xi}, \textbf{-})}
\]
by $D_n^{(m-l+i)}$, dividing its $j$-th column by $\delta_n^{(j)}$ and using 
\eqref{Eq_Sec2.4refinement}.
\end{proof}
\begin{remark}
In Theorem 2.4 we can change the non-vanishing assumption, identical to that 
in Theorem 2.1, by replacing it with the non-vanishing assumption in Theorem 2.2. 
Also, as in the setting of Theorem 2.3 we can enlarge the range for $\nu$ 
allowing $\nu=0,\dots,n$ (i.e.: we have more sequences at our disposal), suppose that
\[
\frac{D_n^{(m)}}{\delta_n^{(\nu_1)}}\cdots\frac{D_n^{(m-l+1)}}{\delta_n^{(\nu_l)}}\,
\det \boldsymbol{\varepsilon}_n^{(\boldsymbol{\mu},\boldsymbol{\nu})} 
													\to 0 \quad (n\to\infty),
\]
and that the non-vanishing assumption in Theorem 2.3 holds. Then the conclusion on the 
dimension of the vector space over $\Q$ spanned by $1,\gamma_1,\dots,\gamma_m$ holds 
all the same.
\end{remark}
\begin{remark}
Theorem 2.4 above is equivalent to its special case where $\delta_n^{\nu}=1$, 
$\nu=1,\dots,l$ (simply put $\widehat{q}_n^{(\nu)}=q_n^{(\nu)}$ 
and $\widehat{p}_n^{(\mu,\nu)}=p_n^{(\mu,\nu)}$). We decided to present 
it in that form in order to stress its meaning in the context outlined 
in the introduction, in which the sequences $D_n^\mu$ represents 
the rough estimate of the denominators of the approximations, while 
the sequences $\delta_n^\nu$ represent the {\it arithmetical correction} arising, 
e.g, from the permutation group method, or from other methods. On the other hand, 
the sequences $\varepsilon_n^{(\mu,\nu)}$ come from a purely analytic construction, 
without any direct consideration of the denominators. Next section is devoted to 
obtaining an estimate of $\det \boldsymbol{\varepsilon}_n^{(\boldsymbol{\mu},
\textbf{-} )}$ in a special situation.
\end{remark}
\begin{remark}
The key player in the above theorem are Grassmann's (or Pl\"ucker's) coordinates 
(i.e.: the maximal order {\it minors}) of the matrix
\begin{equation}				\label{Eq_Sec2.4GrassmannPlucker}
\left[ \begin{matrix} \boldsymbol{q}_n \\ - \\ \boldsymbol{p}_n \end{matrix} 	\right];
\end{equation}
our assumptions just ensure that they become integers, after multiplication 
by $\widehat{D}_n\in\Z$, and, at the same time, $\widehat{D}_n \det 
\boldsymbol{\varepsilon}_n^{(\boldsymbol{\mu},\textbf{-} )} \to 0$. In other words, 
the last theorem implicitly involves an {\it height} of the matrix
\eqref{Eq_Sec2.4GrassmannPlucker}. This height is central in Diophantine geometry: 
see \cite{BombieriVaaler}, and Amoroso's proof of the Nesterenko criterion in \cite{Colmez}. 
Our criterion can be easily extended, as usual, to the linear independence over an imaginary 
quadratic extension of $\Q$, and a generalization of it to arbitrary number fields, 
as in \cite[Proposition 4.1]{Marcovecchio2006}, would arguably involve this height. 

Moreover, it would be interesting to obtain a quantitative version of our criterion, 
yielding a linear independence {\it measure}.
\end{remark}
\section{Minors of the Casoratian matrix}
\subsection{Notation and purpose}
Let $\alpha_n^{(0)},\dots,\alpha_n^{(m)}$ be sequences of complex numbers, and we 
generally assume that $\alpha_n^{(0)}\alpha_n^{(m)}\not=0$. For most applications 
we have in mind, we also require
\begin{equation}			\label{Eq_Sec3.1lim}
\lim_{n\to\infty} \alpha_n^{(j)} = \alpha^{(j)} \quad (j=0,\dots,m),
\end{equation}
in which case we also assume $\alpha_0\alpha_m\not=0$. Let $x_n$ be a sequence 
of complex numbers satisfying 
\begin{equation}			\label{Eq_Sec3.1=0}
\alpha_n^{(m)} x_{n+m} + \alpha_n^{(m-1)} x_{n+m-1} + \cdots + \alpha_n^{(0)} x_n = 0.
\end{equation}
The coefficients $\alpha_n^{(m)}$ and $\alpha_n^{(0)}$ are said to be the {\it highest 
order} and {\it lowest order} coefficients of \eqref{Eq_Sec3.1=0}. It is well known 
that the set of solutions of \eqref{Eq_Sec3.1=0} is a vector space, and that 
\eqref{Eq_Sec3.1=0} can be written as a first order linear recurrence system. 
Given $m$ solutions $x_n^{(1)},\dots,x_n^{(m)}$, they are linearly independent if and 
only if the Casoratian matrix \cite{Casorati}, sometimes also called Wronskian by analogy 
with the differential equation setting,
\[
\boldsymbol{x}_n =  \left[ x_{n+i-1}^{(j)} \right]_{\substack{i=1,\dots,m \\ j=1,\dots,m}}
\]
is non-singular for some $n$ (and therefore for any $n$). In such a case, 
any solution of \eqref{Eq_Sec3.1=0} is a linear combination of $x_n^{(1)},
\dots,x_n^{(m)}$ with constant coefficients (i.e.: independent of $n$), 
and $x_n^{(1)},\dots,x_n^{(m)}$ is said to be a {\it basis of solutions} of 
\eqref{Eq_Sec3.1=0}. Also, $\det \boldsymbol{x}_n$ satisfies the discrete 
Abel formula (see \cite[Problem 2.16.21]{Agarwal})
\[
\det \boldsymbol{x}_{n+1} = (-1)^m \frac{\alpha_n^{(0)}}{\alpha_n^{(m)}} 
\det \boldsymbol{x}_n. 
\]
In other words, $\det \boldsymbol{x}_n$ satisfies a first order linear recurrence 
equation, whose highest and lowest order coefficients are, up to the sign, the highest 
and the lowest coefficient of the linear equation \eqref{Eq_Sec3.1=0}, respectively. 
Hence the coefficients of such a recurrence are independent of the particular basis 
of solutions for \eqref{Eq_Sec3.1=0}. A bit more generally, the coefficients 
$(-1)^{m-r} \alpha_n^{(r)}$, for $r=0,\dots,m$, are easily seen to be proportional to 
\[
\det\left[ x_{n+i}^{(j)} \right]_{\substack{i=0,\dots,\widehat{r},\dots,m \\ j=1,\dots,m}},
\quad r=0,\dots,m,
\]
see \cite[\S 285]{PincherleAmaldi}. We also recall that, regardless of the equation 
\eqref{Eq_Sec3.1=0}, $r$ sequences $z_n^{(1)},\dots,z_n^{(s)}$ are linearly independent 
if and only if 
\[
\det\left[ z_{n+i-1}^{(j)} \right]_{\substack{i=1,\dots,s \\ j=1,\dots,s}}\not=0 
\quad \text{ for infinitely many } n,
\]
see \cite[\S 7]{Casorati}, \cite[\S 279]{PincherleAmaldi}.

The purpose of this section is to study the asymptotic behavior of the $l\times l$ minors 
of $\boldsymbol x_n$. We outline a possible strategy to achieve this goal, which is to 
find a difference equation satisfied by those minors; see \cite[Problem 2.16.23]{Agarwal} 
for the case $l=m-1$ with contiguous rows, while \eqref{Eq_Sec3.1=0} obviously 
copes with the case $l=1$. Then we explain how to circumvent the difficulties 
that arise from that method. Before starting with, we briefly recall the most 
important results by Poincar\'e, Perron and Pituk about the asymptotic behavior 
of solutions of \eqref{Eq_Sec3.1=0} satisfying \eqref{Eq_Sec3.1lim}. 

Minors of the Casoratian (or Wronskian) matrix are key tools in the theory of difference 
(or differential) equations, with regard to disconjugacy, factorization, discrete Rolle 
theorem, and several important results on the same vain: see the milestone paper 
\cite{Hartman}; we refer to \cite[Chapter 10]{Agarwal} for a wide and (relatively) 
updated literature, and to \cite{Coppel} for a nice introduction and some perspectives 
on older results.
\subsection{A short account on Poincar\'e-Perron-Pituk's theorems}
Concerning solutions of \eqref{Eq_Sec3.1=0} with the property \eqref{Eq_Sec3.1lim}, 
Poincar\'e \cite{Poincare} \cite[Theorem 2.14.1]{Agarwal} proved the following: 
if the moduli of the roots $\lambda_1,\dots,\lambda_m$ of the characteristic polynomial 
\[
a^{(m)} \lambda^m + a^{(m-1)} \lambda^{(m-1)} + \cdots + a^{(0)}
\]
are distinct, then either $x_n=0$ for any sufficiently large $n$, or 
\[
\lim_{n\to\infty} \frac{x_{n+1}}{x_n} = \lambda_j \qquad \text{ for some } j=1,\dots,m.
\]
Later on, Perron \cite{Perron1909} \cite[Theorem 2.14.2]{Agarwal} obtained a more precise 
result: there exists a basis $x_n^{(1)},\dots,x_n^{(m)}$ of solutions of
\eqref{Eq_Sec3.1=0} such that 
\[
\lim_{n\to\infty} \frac{x_{n+1}^{(j)}}{x_n^{(j)}} = \lambda_j \qquad \text{ for all } 
j=1,\dots,m.
\]
Also, in the more general situation where the moduli $|\lambda_j|$, and even the roots 
$\lambda_j$ themselves, may coincide, Perron \cite{Perron1909} \cite{Perron1921}  proved 
that there exists a basis of solutions \eqref{Eq_Sec3.1=0} such that 
\[
\limsup_{n\to\infty} \sqrt[n]{|x_n^{(j)}|} = |\lambda_j| \qquad \text{ for all } 
j=1,\dots,m.
\]
In the early 2000's, Pituk \cite{Pituk2002} obtained a new limit relation for all 
non-zero solutions of \eqref{Eq_Sec3.1=0}, without any assumption on the 
muduli $|\lambda_j|$, and even without assuming anything on $a_n^{(0)}$ or $a^{(0)}$:
\[
\lim_{n\to\infty} \sqrt[n]{|x_n|+|x_{n+1}|+\cdots+|x_{n+m-1}|} = |\lambda_j| 
											\qquad \text{ for some } j=1,\dots,m.
\]
The last result is sufficient for certain applications, see \cite{Marcovecchio2014}.
Thus, a rough guess about how to manage with asymptotic behaviors of minors of 
$\boldsymbol x_n$ is to obtain a difference equation for them. It is also worth 
mentioning a theorem by Buslaev refinement \cite{Buslaev} of Poincar\'e's theorem, 
which again does not assume that the roots $\lambda_j$ of the characteristic polynomial 
are distinct in modulus: for any non-zero solution $x_n$ of \eqref{Eq_Sec3.1=0} 
\[
\limsup_{n\to\infty} \sqrt[n]{|x_n|} = |\lambda_j| \qquad \text{ for some } 
j=1,\dots,m,
\]
and $x_n$ satisfies a linear recurrence equation similar to \eqref{Eq_Sec3.1=0}, 
whose monic characteristic polynomial divides the monic characteristic polynomial of 
\eqref{Eq_Sec3.3detdet=0}, and whose characteristic root are all equal in modulus. 
As it was remarked by Zudilin \cite{ZudilinDifference}, this implies that if $x_n$ 
is a non-zero solution of \eqref{Eq_Sec3.1=0} such that 
\[
\limsup_{n\to\infty} \sqrt[n]{|x_n|} = |\lambda_1|,
\]
and if $|\lambda_j|\not=|\lambda_1|$ for $j\not=1$, then 
\[
\lim_{n\to\infty} \frac{x_{n+1}}{x_n} = \lambda_1.
\]
\subsection{The difference equations for the minors}
The criteria in Sect.2 are designed to deal with two different situations: either we have 
$l$ linear independent solutions of \eqref{Eq_Sec3.1=0}, or we have a basis of 
solutions, and select $l$ solution within this basis. Here we unify the exposition: if   
$x_n^{(1)},\dots,x_n^{(m)}$ are $m$ solutions of \eqref{Eq_Sec3.1=0}, for now 
not necessarily linear independent, we select the first $l$ of them and denote
\[
\boldsymbol{x}_n^{(\textbf{-} ,l)} = \left[ x_{n+i-1}^{(j)} 
									\right]_{\substack{i=1,\dots,m \\ j=1,\dots,l}}.
\]
If we take only $l$ solutions of \eqref{Eq_Sec3.1=0} from the very beginning, 
the discussion that follows does not change.

For any $1\leqslant \mu_1 <\cdots <\mu_l\leqslant m$, we pick the square
sub-matrix of $\boldsymbol{x}_n^{(\textbf{-} ,l)}$ with the corresponding rows:
\[
\boldsymbol{x}_n^{(\boldsymbol{\mu},l)} = \left[ x_{n+\mu_i-1}^{(j)} 
									\right]_{\substack{i=1,\dots,l \\ j=1,\dots,l}}.
\]
We wish to find a linear recurrence equation satisfied by 
\[
\det \boldsymbol{x}_n^{(\boldsymbol{\mu},l)}.
\]
To this end, we write
\[
\boldsymbol{x}_{n+1} = \boldsymbol{\Psi}_n \boldsymbol{x}_n,
\]
where 
\[
\boldsymbol{\Psi}_n = \left[ \Psi_n^{(i,j)} \right]_{\substack{i=1,\dots,m \\ j=1,\dots,m}}
\]
is the companion matrix of \eqref{Eq_Sec3.1=0}: 
\[
\Psi_n^{(i,i+1)} = 1 \; (i=1,\dots,m-1), \quad 
\Psi_n^{(m,j)} = -\frac{\alpha_n^{(j-1)}}{\alpha_n^{(m)}} \; (j=1,\dots,m), \quad
\Psi_n^{(i,j)} = 0 \; \text{ otherwise}.
\]
It is worth noticing that
\[
\boldsymbol{\Psi}_n^{-1} = \left[ \breve\Psi_n^{(i,j)} 
							\right]_{\substack{i=1,\dots,m \\ j=1,\dots,m}},
\]
where
\[
\breve\Psi_n^{(1,j)} = -\frac{\alpha_n^{(j)}}{\alpha_n^{(0)}} \; (j=1,\dots,m), \quad
\breve\Psi_n^{(i+1,i)} = 1 \; (i=1,\dots,m-1), \quad 
\breve\Psi_n^{(i,j)} = 0 \; \text{ otherwise}, 
\]
and that if $\lambda_n^{(1)},\dots,\lambda_n^{(n)}$ are the eigenvalues 
of $\boldsymbol\Psi_n$, then $\boldsymbol w_{n+1} = \boldsymbol\Lambda_n \boldsymbol z_n$, 
where
\[
\Lambda_n^{(j,j)} = \lambda_n^{(j)} \; (j=1,\dots,m), 
	\quad \Lambda_n^{(i,i+1)} = 1 \; (i=1,\dots,m-1), \quad  
	\Lambda_n^{(i,j)} = 0 \; \text{ otherwise},
\]
and $\boldsymbol z_n$ and $\boldsymbol w_{n+1}$ are defined by $z_n^{(j)}=x_n^{(j)}$, 
$z_{n+i}^{(j)}=x_{n+i}^{(j)}- \lambda_n^{(i)} x_{n+i-1}^{(j)}$ for $i=1,\dots,m-1$, and 
similarly $w_{n+1}^{(j)}=x_{n+1}^{(j)}$, 
$w_{n+i+1}^{(j)}=x_{n+i+1}^{(j)}- \lambda_n^{(i)} x_{n+i}^{(j)}$ for $i=1,\dots,m-1$. So 
far, we are not assuming that the Casoratian matrix $\boldsymbol{x}_n$ is non-singular. 
Incidentally,
\[
\det \boldsymbol{x}_{n+1} = \det \boldsymbol{\Psi}_n \det \boldsymbol{x}_n,
\]
with
\[
\det \boldsymbol{\Psi}_n = (-1)^n \frac{\alpha_n^{(0)}}{\alpha_n^{(m)}},
\]
gives us the recurrence equation for $\det \boldsymbol{x}_n$ displayed above, 
i.e. it settles the case $l=m$, while \eqref{Eq_Sec3.1=0} obviously copes 
with the case $l=1$. Plainly, $\det \boldsymbol x_n = \det \boldsymbol z_n 
= \det \boldsymbol w_n$ 
and $\det \boldsymbol\Psi_n=\det \boldsymbol\Lambda_n=\lambda_n^{(1)}\dots\lambda_n^{(m)}$. 
 
By induction on $k$,
\[
\boldsymbol{x}_{n+k} = \boldsymbol{\Psi}_{n+k-1}\cdots\boldsymbol{\Psi}_n \boldsymbol{x}_n,
\]
hence
\[
\boldsymbol{x}_{n+k}^{(\boldsymbol{\mu},l)} = 
		\left[ \boldsymbol{\Psi}_{n+k-1}\cdots
				\boldsymbol{\Psi}_n \right]^{(\boldsymbol{\mu}, \textbf{-} )}
	\boldsymbol{x}_n^{( \textbf{-} ,l)}.
\]
Here and hereafter, the empty product of matrices is the identity matrix. By the 
Binet-Cauchy formula,
\[
\det \boldsymbol{x}_{n+k}^{(\boldsymbol{\mu},l)} = \sum_{\boldsymbol{\nu}} 
		\det \left[ \boldsymbol{\Psi}_{n+k-1}\cdots
				\boldsymbol{\Psi}_n \right]^{(\boldsymbol{\mu},\boldsymbol{\nu})}
		\det \boldsymbol{x}_n^{( \boldsymbol{\nu}, l)},
\]
where the sum is over all $\boldsymbol{\nu}=(\nu_1,\dots,\nu_l)$ such that 
$1\leqslant \nu_1<\cdots< \nu_l\leqslant m$. 

For a fixed $\boldsymbol{\mu}$, and for each $k$, we consider the coordinates 
of $\det \boldsymbol{x}_{n+k}^{(\boldsymbol{\mu},l)}$ with respect 
to $\det\boldsymbol{x}_n^{( \boldsymbol{\nu}, l)}$, where $\boldsymbol{\nu}$ 
varies: these are precisely $\det \left[ \boldsymbol{\Psi}_{n+k-1}\cdots
				\boldsymbol{\Psi}_n \right]^{(\boldsymbol{\mu},\boldsymbol{\nu})}$. 
Thus, by letting $k$ vary from $0$ to $\binom{m}{l}$, it is straightforward to 
obtain a linear difference equation of order $\binom{m}{l}$ satisfied 
by $y_n=\det \boldsymbol{x}_n^{(\boldsymbol{\mu},l)}$:
\begin{equation}			\label{Eq_Sec3.3detdet=0}
\sum_{k=0}^{\binom{m}{l}} (-1)^k \det \left[
					\det \left[ \boldsymbol{\Psi}_{n+j-1}\cdots
				\boldsymbol{\Psi}_n \right]^{(\boldsymbol{\mu},\boldsymbol{\nu})}
	\right]_{\substack{j=0,\dots,\widehat{k},\dots,\binom{m}{l} 
					\\ 1\leqslant\nu_1< \cdots< \nu_l\leqslant m}} 
		\;	y_{n+k}= 0,
\end{equation}
where $\widehat{k}$ means that the index $k$ is omitted in the range for $j$. Note 
that any minor $y_n=\det \boldsymbol{x}_n^{(\boldsymbol{\mu},\boldsymbol\upsilon)}$  
is a solution of \eqref{Eq_Sec3.3detdet=0}, because the coefficients 
in \eqref{Eq_Sec3.3detdet=0} do not depend on the choice of the 
columns $\upsilon_1,\dots,\upsilon_l$, while they do depend on the 
choice of the rows $\mu_1,\dots,\mu_l$. Also, in \eqref{Eq_Sec3.3detdet=0} and in similar 
formulas below, unless otherwise stated, we can take any ordering in the set 
for $\boldsymbol\nu$ (of course, the same each time, in the same formula).

Again keeping $\boldsymbol\mu$ fixed, the $\binom{m}{l}$ solutions 
$y_n=\det \boldsymbol{x}_n^{(\boldsymbol{\mu},\boldsymbol\upsilon)}$ 
of \eqref{Eq_Sec3.3detdet=0} found above, 
with $1\leqslant \upsilon_1<\cdots<\upsilon_l\leqslant m$, 
are linearly independent if and only if the Casoratian matrix
\[
\left[ \det \boldsymbol{x}_{n+j-1}^{(\boldsymbol{\mu},\boldsymbol\upsilon)} 
		\right]_{\substack{j=1,\dots,\binom{m}{l} \\ 
			1\leqslant \upsilon_1<\cdots<\upsilon_l\leqslant m}}
\]
is non-singular. Just as above, we have 
\begin{multline*}
\det \left[ \det \boldsymbol{x}_{n+j-1}^{(\boldsymbol{\mu},\boldsymbol\upsilon)} 
		\right]_{\substack{j=1,\dots,\binom{m}{l} \\ 
			1\leqslant \upsilon_1<\cdots<\upsilon_l\leqslant m}} = 
\det \left[ \sum_{\boldsymbol{\nu}} 
		\det \left[ \boldsymbol{\Psi}_{n+j-2}\cdots
				\boldsymbol{\Psi}_n \right]^{(\boldsymbol{\mu},\boldsymbol{\nu})}
		\det \boldsymbol{x}_n^{( \boldsymbol{\nu}, \boldsymbol{\upsilon})}
		\right]_{\substack{j=1,\dots,\binom{m}{l} \\ 
			1\leqslant \upsilon_1<\cdots<\upsilon_l\leqslant m}} \\
= \det \left[ \det \left[ \boldsymbol{\Psi}_{n+j-2}\cdots
				\boldsymbol{\Psi}_n \right]^{(\boldsymbol{\mu},\boldsymbol{\nu})}
		\right]_{\substack{j=1,\dots,\binom{m}{l} \\ 
			1\leqslant \nu_1<\cdots<\nu_l\leqslant m}}
  \det \Big[ \det \boldsymbol{x}_n^{( \boldsymbol{\nu}, \boldsymbol{\upsilon})}
		\Big]_{\substack{ 1\leqslant \nu_1<\cdots<\nu_l\leqslant m \\ 
			1\leqslant \upsilon_1<\cdots<\upsilon_l\leqslant m}},
\end{multline*}
where the Binet formula $\det AB = \det A \det B$ was used. Here, one more time, 
the ordering in the set for $\boldsymbol\nu$ (resp. for $\boldsymbol\upsilon$) must 
be the same at each occurrence, while it needs not to be identical for $\boldsymbol\nu$ 
{\it and } $\boldsymbol\upsilon$, though it would be more consistent; on the other hand, 
there is no way to choose the same ordering for $j$ and $\boldsymbol{\nu}$ (that would 
just be non-sense). By \eqref{Eq_Sec3.5GrassmannDeterminant} below,
\[
\det \left[ \det \boldsymbol{x}_{n+j-1}^{(\boldsymbol{\mu},\boldsymbol\upsilon)} 
		\right]_{\substack{j=1,\dots,\binom{m}{l} \\ 
			1\leqslant \upsilon_1<\cdots<\upsilon_l\leqslant m}} =  
\det \left[ \det \left[ \boldsymbol{\Psi}_{n+j-2}\cdots
				\boldsymbol{\Psi}_n \right]^{(\boldsymbol{\mu},\boldsymbol{\nu})}
		\right]_{\substack{j=1,\dots,\binom{m}{l} \\ 
			1\leqslant \nu_1<\cdots<\nu_l\leqslant m}} 
	\left( \det \boldsymbol{x}_n \right)^{\binom{m-1}{l-1}}.
\]
We remark that the highest and the lowest order coefficients in \eqref{Eq_Sec3.3detdet=0}, 
respectively, are
\[
(-1)^{\binom{m}{l}} \det \left[ \det \left[ \boldsymbol{\Psi}_{n+j-2}\cdots
				\boldsymbol{\Psi}_n \right]^{(\boldsymbol{\mu},\boldsymbol{\nu})}
		\right]_{\substack{j=1,\dots,\binom{m}{l} \\ 
			1\leqslant \nu_1<\cdots<\nu_l\leqslant m}} 
\]
and, using the Binet formula and \eqref{Eq_Sec3.5GrassmannDeterminant} again, 
\begin{multline}		\label{Eq_Sec3.4BinetAbel}
 \det \left[ \det \left[ \boldsymbol{\Psi}_{n+j-1}\cdots
				\boldsymbol{\Psi}_n \right]^{(\boldsymbol{\mu},\boldsymbol{\nu})}
		\right]_{\substack{j=1,\dots,\binom{m}{l} \\ 
			1\leqslant \nu_1<\cdots<\nu_l\leqslant m}} 	\\	= 
 \det \left[ \det \left[ \boldsymbol{\Psi}_{n+j-1}\cdots
				\boldsymbol{\Psi}_{n+1} \right]^{(\boldsymbol{\mu},\boldsymbol{\nu})}
		\right]_{\substack{j=1,\dots,\binom{m}{l} \\ 
			1\leqslant \nu_1<\cdots<\nu_l\leqslant m}}
		\left( \det \boldsymbol{\Psi}_n \right)^{\binom{m-1}{l-1}},
\end{multline}
in accordance with the discrete Abel formulas for \eqref{Eq_Sec3.1=0} and 
for \eqref{Eq_Sec3.3detdet=0}.

Our conclusion, for this subsection, reads as follows: if the quantity in \eqref{Eq_Sec3.4BinetAbel} 
is non-zero for some $n$ (thus is so for any $n$), and if $x_n^{(1)},\dots,x_n^{(m)}$ is a 
basis of solutions of \eqref{Eq_Sec3.1=0}, then 
$\{\det \boldsymbol{x}_n^{(\boldsymbol{\mu},\boldsymbol\upsilon)}: 
1\leqslant \upsilon_1<\cdots<\upsilon_l\leqslant m\}$ is a basis of 
solutions of \eqref{Eq_Sec3.3detdet=0}.
\bigskip
\begin{remark}
The coefficients of the recurrence equation \eqref{Eq_Sec3.3detdet=0} only depend on the 
coefficients of the recurrence \eqref{Eq_Sec3.1=0}, and do not depend on a 
basis of solutions, nor on a choice for the columns of the minor. To be more 
precise, since the coefficients of the matrix $\boldsymbol\Psi_n$ are either $0$ or, up 
to the sign, an elementary symmetric function in the eigenvalues $\lambda_n^{(1)},
\dots,\lambda_n^{(m)}$, 
\[
\text{Sym}^{(r)} (\lambda_n^{(1)},\dots,\lambda_n^{(m)}) 
	= \sum\limits_{1\leqslant \nu_1<\cdots<\nu_r\leqslant m} 
	\lambda_n^{(\nu_1)}\cdots\lambda_n^{(\nu_r)}, \quad r=0,\dots,m,
\]
where $\text{Sym}^{(0)} (\lambda_n^{(1)},\dots,\lambda_n^{(m)})=1$, for 
all $\boldsymbol\nu$ with $1\leqslant \nu_1<\cdots<\nu_l\leqslant m$ we have universal 
polynomials in $z_{i,r}$, with $i=1,\dots,\binom{m}{l}-1$, $r=1,\dots,m$, with integer 
coefficients and partial degree not exceeding $1$ in each of $z_{i,r}$, such that 
their values at $z_{i,r}=\text{Sym}^{(r)} (\lambda_{n+i}^{(1)},\dots,\lambda_{n+i}^{(m)})$  
are the coefficients of the equation \eqref{Eq_Sec3.3detdet=0}. 
\end{remark}
\begin{remark}
We stress that the lowest and highest order coefficients in \eqref{Eq_Sec3.1=0} and 
in \eqref{Eq_Sec3.3detdet=0} are related by \eqref{Eq_Sec3.4BinetAbel}, and that the lowest order 
coefficient for a given $n$ is, up to a non-zero constant, the highest order coefficient 
for $n+1$. For this reason, it it sufficient to check the non-vanishing of one of the 
two (say: the highest order coefficient) for any $n$, in order to apply the described 
method.
\end{remark}
\begin{remark}
There is an equivalent way to get the recurrence \eqref{Eq_Sec3.3detdet=0}, that we 
outline here. Let $\lambda_n^{(1)},\dots,\lambda_n^{(m)}$ be the (non-zero) roots of 
\[
\alpha_n^{(m)}\lambda^m + \alpha_n^{(m-1)}\lambda^{m-1}+\cdots+\alpha_n^{(0)} = 0,
\]
which, essentially, may be supposed to be distinct, as we are going to see. Then the 
rows of the $(m+1)\times (m+1)$ matrix
\[
\left[ \left[ x_{n+i}^{(j)} \right]_{\substack{i=0,\dots,m \\ j=1,\dots,l}} \vert 
	\left[ {\lambda_n^{(\nu_j)}}^i \right]_{\substack{i=0,\dots,m \\ j=l,\dots,m}} \right]
\]
are linearly dependent, so that its determinant vanishes. Here, $\nu_l,\dots,\nu_m$ 
are arbitrarily chosen indices with $1\leqslant \nu_l<\cdots<\nu_m\leqslant m$, so that 
we have $\binom{m}{l-1}$ such vanishing determinants, for each $n$. Each determinant can 
be expanded with the help of Laplace formula along the first $l$ columns, to obtain
\begin{equation} 			\label{Eq_Sec3.4Laplace}
\sum\limits_{0\leqslant \mu_1<\cdots<\mu_l\leqslant m} (-1)^{|\boldsymbol\mu|} 
	\det \left[ x_{n+i}^{(j)} \right]_{\substack{i=\mu_1,\dots,\mu_l \\ j=1,\dots,l}}
	\det \left[ {\lambda_n^{(\nu_j)}}^i \right]_{
			\substack{i=\widehat{\mu}_l,\dots,\widehat{\mu}_m \\ j=l,\dots,m}} = 0,
\end{equation}
where $|\boldsymbol\mu|=\mu_1+\cdots+\mu_l$, and $\widehat{\mu}_l,\dots,\widehat{\mu}_m$ 
are the complementary indices of $\mu_1,\dots,\mu_l$ in $0,\dots,m$. Thus, each sum 
contains $\binom{m+1}{l}=\binom{m}{l}+\binom{m}{l-1}$ terms, note, however, that only 
$2\binom{m}{l}-\binom{m-1}{l}=\binom{m}{l}+\binom{m-1}{l-1}$ of them have a minor of the 
Casoratian matrix as a factor. By considering consecutive values for $n$, we have 
only $\binom{m}{l-1}$ {\it new} terms, where {\it new} refers to their $x$-determinant 
factor, for each new value of $n$, and the same number of new equations that correspond 
to different choices of $\boldsymbol\nu$. Thus, for a fixed $\boldsymbol\mu$ 
with $1\leqslant \mu_1<\dots<\mu_l\leqslant m$, taking a linear combination of 
\eqref{Eq_Sec3.4Laplace} for $n,n+1,\dots,n+\binom{m}{l}-1$, we get a vanishing 
linear combination of terms of the type
\[
\det \left[ x_{n+i-1}^{(j)} \right]_{\substack{i=\mu_1,
										\dots,\mu_l \\ j=1,\dots,l}}
\] 
only, for $n,n+1,\dots,n+\binom{m}{l}$. Finally, we observe that 
each equation \eqref{Eq_Sec3.4Laplace} can be divided by $\text{Vandermonde } 
(\lambda_{\nu_l},\dots,\lambda_{\nu_m})$, and after this operation 
the $\lambda$-determinant factors in \eqref{Eq_Sec3.4Laplace} are replaced with polynomials 
in $\text{Sym}^{(r)}(\lambda_n^{(l)},\dots,\lambda_n^{(m)})$, and we do not need to 
assume that $\lambda_n^{(1)},\dots,\lambda_n^{(m)}$ are distinct. 
\end{remark}
\begin{remark}
Seemingly, yet another way to obtain the recurrence equation \eqref{Eq_Sec3.3detdet=0} is 
by induction on $m-l$, using the condensation formula \cite[(2.16)]{Krattenthaler}. 
\end{remark}
\subsection{Recurrences with constant coefficients}
Let us consider the special case when the coefficients of the equation 
\eqref{Eq_Sec3.1=0} are independent of $n$:
\begin{equation}			\label{Eq_Sec3.4independentofn}
\alpha^{(m)} x_{n+m} + \alpha^{(m-1)} x_{n+m-1} + \cdots + \alpha^{(0)} x_n = 0,
\end{equation}
and suppose that $\alpha^{(0)}\alpha^{(m)}\not=0$. If the roots $\lambda_1,\dots,
\lambda_m$ of the polynomial
\begin{equation}			\label{Eq_Sec3.4polynomial}
\alpha^{(m)} \lambda^{n+m} + \alpha^{(m-1)} \lambda^{n+m-1} + \cdots + \alpha^{(0)} = 0
\end{equation}
are distinct, then $x_n^{(j)}=\lambda_j^n$ ($j=1,\dots,m$) is 
a basis of solutions of \eqref{Eq_Sec3.4independentofn}, because 
\begin{equation}			\label{Eq_Sec3.4Vandermonde}
\det \left[ x_{n+i-1}^{(j)} \right]_{\substack{i=1,\dots,m \\ j=1,\dots,m}} = 
\det \left[ \lambda_j^{n+i-1} \right]_{\substack{i=1,\dots,m \\ j=1,\dots,m}} = 
(\lambda_1\cdots\lambda_m)^n \prod_{1\leqslant i<j\leqslant m} (\lambda_j-\lambda_i) 
\not = 0.
\end{equation}
The columns of the matrix in \eqref{Eq_Sec3.4Vandermonde} are the eigenvectors 
of the companion matrix 
\begin{equation}			\label{Eq_Sec3.4companionMatrix}
\boldsymbol{\Psi} = \left[ \Psi^{(i,j)} \right]_{\substack{i=1,\dots,m \\ j=1,\dots,m}}
\end{equation}
of the recurrence equation \eqref{Eq_Sec3.4independentofn}, defined by
\[
\Psi^{(i,i+1)} = 1 \; (i=1,\dots,m-1); \quad 
\Psi^{(m,j)} = -\frac{\alpha^{(j-1)}}{\alpha^{(m)}} \; (j=1,\dots,m); \quad
\Psi^{(i,j)} = 0 \; \text{ otherwise},
\]
so that
\[
\boldsymbol{\Psi x}_n = \boldsymbol x_n \boldsymbol \Delta, 
\]
where $\boldsymbol\Delta = {\rm diag }(\lambda_1,\dots,\lambda_m)$.

This holds in particular when $\lambda_1,\dots,\lambda_m$ additionally satisfy
\[
\lambda_{i+1}-\lambda_i = \varepsilon \quad 
	\text{ for any }\; i\not=k_1,k_1+k_2,\dots,k_1+\cdots+k_{r-1},
\]
where $k_1+\cdots+k_r=m$. After a few elementary manipulations on the columns 
of $\boldsymbol x_n$, dividing by a suitable power of $\varepsilon$ and making 
$\varepsilon\to 0$ (keeping the $r$ numbers $\lambda_{k_1},\lambda_{k_1+k_2},\dots,
\lambda_{k_1+\cdots+k_r}$ fixed) in \eqref{Eq_Sec3.4Vandermonde}, we obtain, by changing 
the notation, a basis of solutions of \eqref{Eq_Sec3.4independentofn}, 
which also is a basis of eigenvectors of $\boldsymbol\Psi$, 
when $\lambda_1,\dots,\lambda_r$ are the distinct roots of 
\eqref{Eq_Sec3.4polynomial} with multiplicities $k_1,\dots,k_r$:
\begin{multline*}
x_n^{(k)} = \binom{n}{k-1} \lambda_1^{n-k+1} \quad (k=1,\dots,k_1),\quad 
x_n^{(k_1+k)} = \binom{n}{k-1} \lambda_2^{n-k+1} \quad (k=1,\dots,k_2), \dots \\ \dots, 
x_n^{(k_1+\cdots+k_{r-1}+k)} = \binom{n}{k-1} \lambda_r^{n-k+1} \quad (k=1,\dots,k_r),
\end{multline*}
because now
\[
\det \left[ x_{n+i-1}^{(j)} \right]_{\substack{i=1,\dots,m \\ j=1,\dots,m}} = 
			(\lambda_1^{k_1}\cdots\lambda_r^{k_r})^n 
			\prod_{1\leqslant i<j\leqslant r} (\lambda_j-\lambda_i)^{k_i k_j} \not = 0,
\]
see \cite[pp.174--176]{Meray}, or, for a modern and well informed source, 
\cite[Theorem 20]{Krattenthaler}. 

Let us suppose, to avoid complications, that $\lambda_1,\dots,\lambda_m$ distinct. 
We may apply the arguments in Sect. 3.3, and find a difference equation for the minors
\[
y_n = \det \boldsymbol x_n^{\boldsymbol\mu,\boldsymbol\upsilon} = 
	\det \left[ \lambda_{\upsilon_j}^{n+\mu_i-1} 
		\right]_{\substack{i=1,\dots,l \\ j=1,\dots,l}} = 
	(\lambda_{\upsilon_1}\cdots\lambda_{\upsilon_l})^n 
	\det \left[ \lambda_{\upsilon_j}^{\mu_i-1} 
		\right]_{\substack{i=1,\dots,l \\ j=1,\dots,l}},
\] 
namely
\begin{equation}			\label{Eq_Sec3.4namelydetdet=0}
\sum_{k=0}^{\binom{m}{l}} (-1)^k \det 
	\left[	\det \left[ \boldsymbol{\Psi}^j \right]^{(\boldsymbol{\mu}\boldsymbol{\nu})}
	\right]_{\substack{j=0,\dots,\widehat{k},\dots,\binom{m}{l} 
					\\ 1\leqslant\nu_1< \cdots< \nu_l\leqslant m}} 
		\;	y_{n+k}= 0.
\end{equation}
If 
\[
\det \boldsymbol x_0^{\boldsymbol\mu,\boldsymbol\upsilon} = 
\det \left[ \lambda_{\upsilon_j}^{\mu_i-1} 
		\right]_{\substack{i=1,\dots,l \\ j=1,\dots,l}} \not = 0
\]
for all $\boldsymbol\upsilon$, then by \eqref{Eq_Sec3.4namelydetdet=0} 
the $\binom{m}{l}$ products 
$\lambda_{\boldsymbol\upsilon} = \lambda_{\upsilon_1} \cdots \lambda_{\upsilon_l}$, for 
$1\leqslant \upsilon_1 < \cdots <\upsilon_l \leqslant m$, are roots of the polynomial
\begin{equation}			\label{Eq_Sec3.4detdetlambdak}
\sum_{k=0}^{\binom{m}{l}} (-1)^k \det \left[
					\det \left[ \boldsymbol{\Psi}^j \right]^{(\boldsymbol{\mu},
															\boldsymbol{\nu})}
	\right]_{\substack{j=0,\dots,\widehat{k},\dots,\binom{m}{l} 
					\\ 1\leqslant\nu_1< \cdots< \nu_l\leqslant m}} 
		\; \lambda^k.
\end{equation}
In addition, if $\lambda_{\boldsymbol\upsilon}$ are all distinct, 
then the highest and lowest (see remark 3.2) coefficients of 
\eqref{Eq_Sec3.4detdetlambdak} are non-zero, because
\begin{multline*}
\det \left[
					\det \left[ \boldsymbol{\Psi}^j \right]^{(\boldsymbol{\mu},\boldsymbol{\nu})}
	\right]_{\substack{j=0,\dots,\binom{m}{l}-1 
					\\ 1\leqslant\nu_1< \cdots< \nu_l\leqslant m}}
\det \Big[ \det \boldsymbol x_0^{\boldsymbol\nu,\boldsymbol\upsilon} \Big]_{ 
\substack{1\leqslant\nu_1< \cdots< \nu_l\leqslant m \\ 
			1\leqslant\upsilon_1< \cdots< \upsilon_l\leqslant m}}
\\ =
\det \left[			\det \left[ \boldsymbol{\Psi}^j \boldsymbol x_0 				
						\right]^{(\boldsymbol{\mu},\boldsymbol{\upsilon})}
	\right]_{\substack{j=0,\dots,\binom{m}{l}-1 
					\\ 1\leqslant\upsilon_1< \cdots< \upsilon_l\leqslant m}}
 	=
\det \left[			\det \left[ \boldsymbol x_0 \boldsymbol{\Delta}^j				
						\right]^{(\boldsymbol{\mu},\boldsymbol{\upsilon})}
	\right]_{\substack{j=0,\dots,\binom{m}{l}-1 
					\\ 1\leqslant\upsilon_1< \cdots< \upsilon_l\leqslant m}} 
\\ = 
\det \Big[ \det \boldsymbol x_0^{\boldsymbol\mu,\boldsymbol\upsilon} 
					\lambda_{\boldsymbol\upsilon}^j 
	\Big]_{ \substack{ j=0,\dots,\binom{m}{l}-1 \\ 
			1\leqslant\upsilon_1< \cdots< \upsilon_l\leqslant m}} 
= \prod\limits_{1\leqslant\upsilon_1< \cdots< \upsilon_l\leqslant m} 
			\det \boldsymbol x_0^{\boldsymbol\mu,\boldsymbol\upsilon} \cdot  
\det \big[ \lambda_{\boldsymbol\upsilon}^j 
	\big]_{ \substack{ j=0,\dots,\binom{m}{l}-1 \\ 
			1\leqslant\upsilon_1< \cdots< \upsilon_l\leqslant m}} \not = 0 
\end{multline*}
and
\[
\det \Big[ \det \boldsymbol x_0^{\boldsymbol\nu,\boldsymbol\upsilon} \Big]_{ 
\substack{1\leqslant\nu_1< \cdots< \nu_l\leqslant m \\ 
			1\leqslant\upsilon_1< \cdots< \upsilon_l\leqslant m}}
=  \left( \det \boldsymbol{x}_0 \right)^{\binom{m-1}{l-1}} \not = 0.
\]
Putting this in a different way, by performing the previous trick for all the coefficients 
of the polynomial \eqref{Eq_Sec3.4detdetlambdak}, we see that, under the non-vanishing 
assumption  $\det \boldsymbol x_0^{\boldsymbol\mu,\boldsymbol\upsilon}\not=0$ for all 
$\boldsymbol\upsilon$, the polynomial \eqref{Eq_Sec3.4detdetlambdak} is a 
multiple (by a non-zero coefficient) of
\[
\sum_{k=0}^{\binom{m}{l}} (-1)^k 
\det \big[ \lambda_{\boldsymbol\upsilon}^j 
	\big]_{ \substack{ j=0,\dots,\widehat{k},\dots,\binom{m}{l} \\ 
			1\leqslant\upsilon_1< \cdots< \upsilon_l\leqslant m}}\, \lambda^k 
= {\rm Vandermonde} (\lambda_{\boldsymbol\upsilon}: \boldsymbol\upsilon)
\prod\limits_{1\leqslant\upsilon_1< \cdots< \upsilon_l\leqslant m} 
				(\lambda-\lambda_{\boldsymbol\upsilon}),
\]
where
\[
{\rm Vandermonde} (\lambda_{\boldsymbol\upsilon}: \boldsymbol\upsilon) = 
\det \big[ \lambda_{\boldsymbol\upsilon}^j 
	\big]_{ \substack{ j=0,\dots,\binom{m}{l}-1 \\ 
			1\leqslant\upsilon_1< \cdots< \upsilon_l\leqslant m}} \not = 0.
\]
Moreover, the Casoratian matrix 
\[
\left[\det \boldsymbol x_{n+j-1}^{\boldsymbol\mu,\boldsymbol\upsilon} 
\right]_{\substack{
		j=0,\dots,\binom{m}{l} \\ 1\leqslant \upsilon_1<\cdots< \upsilon_l\leqslant m }}
\]
is non-singular, because
\[
\det \left[\det \boldsymbol x_{n+j-1}^{\boldsymbol\mu,\boldsymbol\upsilon} 
\right]_{\substack{
		j=0,\dots,\binom{m}{l} \\ 1\leqslant \upsilon_1<\cdots< \upsilon_l\leqslant m }} 
= \prod\limits_{1\leqslant\upsilon_1< \cdots< \upsilon_l\leqslant m} 
			\det \boldsymbol x_0^{\boldsymbol\mu,\boldsymbol\upsilon} \cdot  
\det \big[ \lambda_{\boldsymbol\upsilon}^j 
	\big]_{ \substack{ j=0,\dots,\binom{m}{l}-1 \\ 
			1\leqslant\upsilon_1< \cdots< \upsilon_l\leqslant m}} \not = 0.
\]

Clearly, if $\lambda_{\boldsymbol\upsilon}$ are distinct, then {\it a fortiori} 
$\lambda_j$ are distinct. Note, however, that $y_n$ satisfies the difference equation
\[
y_{n+q} + \beta^{(q-1)} y_{n+q-1}+\cdots+\beta^{(0)} = 0,
\]
for $q=\binom{m}{l}$ and $\beta^{(0)},\dots,\beta^{(q-1)}$ defined by
\[
\prod\limits_{\boldsymbol\upsilon} (\lambda-\lambda_{\boldsymbol\upsilon}) = 
\lambda^q + \beta^{(q-1)} \lambda^{q-1}+\cdots+\beta^{(0)},
\]
regardless to whether $\lambda_{\boldsymbol\upsilon}$ are distinct or not. If they are 
not distinct, the minors $\det \boldsymbol x_n^{\boldsymbol\mu,\boldsymbol\upsilon}$ are 
no longer a basis of solutions of the recurrence \eqref{Eq_Sec3.4namelydetdet=0}. 

It is fairly possible that in concrete applications of the outlined method in the 
environment of our criteria in Sect.2, the assumption that $\lambda_{\boldsymbol\nu}$ 
are distinct is fulfilled. In this case, one can deal with the requirement that 
\eqref{Eq_Sec3.4BinetAbel} does not vanish, by combining the above discussion with \eqref{Eq_Sec3.1lim}, 
and recalling that $\alpha^{(0)}\alpha^{(m)}\not=0$. However, we seek for more generality, 
specially because the result that we present looks like much more ready-to-use than 
the recurrence \eqref{Eq_Sec3.3detdet=0}. On the other hand, in some cases one may wish to
apply Buslaev's theorem and Zudilin's corollary described above to \eqref{Eq_Sec3.3detdet=0}, 
which therefore is of some interest by itself. 
\subsection{The Sylvester-Franke Theorem}
The following fundamental result in the theory of determinants made its appearance 
in Sect. 3.3, and is crucial in rest of this section.
\begin{theorem}{(Sylvester-Franke's theorem \cite{Sylvester} \cite{Franke})}
Let $\boldsymbol\rho$ be a $m\times m$ matrix with entries in $\C$, and 
let $\lambda_1,\dots,\lambda_m$ be the eigenvalues of $\boldsymbol\rho$, 
repeated with their algebraic multiplicity. Then for all $l=1,\dots,m$ 
the eigenvalues of the matrix
\begin{equation}				\label{Eq_Sec3.5GrassmanMatrix}
\left[ \det \boldsymbol\rho^{(\boldsymbol\mu, \boldsymbol\nu)} 
		\right]_{\substack{1\leqslant \mu_1<\dots<\mu_l\leqslant m \\ 
							1\leqslant \nu_1<\dots<\nu_l\leqslant m}},
\end{equation}
whose $\binom{m}{l}$ rows and columns are arranged with the same (say: the 
lexicographical) ordering, are $\lambda_{\boldsymbol\upsilon}=\lambda_{\upsilon_1}
\cdots\lambda_{\upsilon_l}$, for $1\leqslant\upsilon_1<\cdots<\upsilon_l\leqslant m$, 
again repeated with their algebraic multiplicity. 

In particular  
\begin{equation}				\label{Eq_Sec3.5GrassmannDeterminant}
\det \left[ \det \boldsymbol\rho^{(\boldsymbol\mu, \boldsymbol\nu)} 
		\right]_{\substack{1\leqslant \mu_1<\dots<\mu_l\leqslant m \\ 
							1\leqslant \nu_1<\dots<\nu_l\leqslant m}} 
= \left( \det \boldsymbol\rho \right)^{\binom{m-1}{l-1}}. 
\end{equation}
\end{theorem}
\begin{proof}
Let us prove \eqref{Eq_Sec3.5GrassmannDeterminant}, up to the sign, and under the 
assumption $\det \boldsymbol \rho \not = 0$. Up to reordering the columns (or the rows) 
in $\boldsymbol \rho$ we may suppose that all principal minors $\det\boldsymbol 
\rho^{\left((1,\dots,k),(1,\dots,k)\right)}$, for $k=1,\dots,m$, are non-zero. In 
this setting one could even determine all the eigenvalues of \eqref{Eq_Sec3.5GrassmanMatrix}, 
and, as a result, obtain \eqref{Eq_Sec3.5GrassmanMatrix}; note, however, that the eigenvalues 
may change because of the permutation of the rows (or of the columns) in $\boldsymbol\rho$. 
By assumptions, there exist an upper triangular matrix $\boldsymbol U$, with $1$'s on 
its diagonal, a lower triangular matrix $\boldsymbol L$ with $1$'s on its diagonal, and 
a diagonal matrix $\boldsymbol \Delta$ with the $\lambda_j$'s on its diagonal (possibly 
up to a permutation), such that 
\[
\boldsymbol \rho = \boldsymbol L \boldsymbol \Delta \boldsymbol U. 
\]
By the Binet-Cauchy formula applied twice,
\[
\Big[ \det \boldsymbol\rho^{( \boldsymbol\mu, \boldsymbol\nu)}
		\Big]_{\substack{  \boldsymbol\mu \\ \boldsymbol\nu }}   = 
\Big[ \det \boldsymbol{L}^{( \boldsymbol\mu, \boldsymbol\tau)}
		\Big]_{\substack{ \boldsymbol\mu \\ \boldsymbol\tau }}
\Big[ \det \boldsymbol{\Delta}^{( \boldsymbol\tau, \boldsymbol\upsilon)}
		\Big]_{\substack{ \boldsymbol\tau \\ \boldsymbol\upsilon }}
\Big[ \det \boldsymbol{U}^{( \boldsymbol\upsilon, \boldsymbol\nu)}
		\Big]_{\substack{ \boldsymbol\upsilon \\ \boldsymbol\nu }},
\]
where the ranges for $\boldsymbol\mu$, $\boldsymbol\tau$, 
$\boldsymbol\upsilon$ and $\boldsymbol\nu$ are the same, and the lexicographical 
ordering is chosen at any occurrence of each multi-index. The last formula displays 
a product of three $\binom{m}{l}\times\binom{m}{l}$ matrices, namely: a lower 
triangular matrix with $1$'s on the diagonal, a diagonal matrix with the 
products $\lambda_{\boldsymbol\nu}$ on its diagonal, and an upper 
triangular matrix with $1$'s on the diagonal. Therefore 
\[
\det \Big[ \det \boldsymbol\rho^{( \boldsymbol\mu, \boldsymbol\nu)}
		\Big]_{\substack{  \boldsymbol\mu \\ \boldsymbol\nu }} =
\det \Big[ \det \boldsymbol{\Delta}^{( \boldsymbol\tau, \boldsymbol\upsilon)}
		\Big]_{\substack{ \boldsymbol\tau \\ \boldsymbol\upsilon }}
 = \left( \det \boldsymbol{\Delta} \right)^{\binom{m-1}{l-1}} 
				= \left( \det \boldsymbol\rho \right)^{\binom{m-1}{l-1}},
\]
as we claimed. In particular, if $\boldsymbol\rho$ is non-singular, then its $l$-th 
compound matrix \eqref{Eq_Sec3.5GrassmanMatrix} is also non-singular.  It could be seen 
that the products $\lambda_{\boldsymbol\nu}$, which are the eigenvalues of the $l$-th 
compound matrix of $\boldsymbol\Delta$, are also the eigenvalues of 
\eqref{Eq_Sec3.5GrassmanMatrix}, which would imply our claim in this 
special case, but we are about to prove it in general.

We now prove that $\lambda_{\boldsymbol\nu}$ 
are the eigenvalues of \eqref{Eq_Sec3.5GrassmanMatrix} without assuming neither 
$\det \boldsymbol \rho \not = 0$, nor the non-vanishing of the principal minors 
of $\boldsymbol \rho$. In $\C$ a passage to the limit would suffice, but we prefer 
an algebraic proof. Also, the abstract argument in \cite{Flanders} seemingly 
requires using the axiom of choice, which we do not require here. 

Let $\boldsymbol\sigma$ be a non-singular square matrix such that
\[
\boldsymbol\rho \boldsymbol\sigma = \boldsymbol\sigma \boldsymbol J,
\]
where $\boldsymbol J$ is Jordan's canonical form of $\boldsymbol\rho$, so 
that $\boldsymbol J = \boldsymbol \Lambda + \boldsymbol\Omega$, where 
$\boldsymbol \Lambda$ is a diagonal matrix with $\lambda_1,\dots,\lambda_m$ 
on its diagonal (up to the order), and $\boldsymbol\Omega$ is a nilpotent, 
strictly upper (according to some authors, lower), triangular matrix. Just 
as above, we have
\[
\Big[ \det \boldsymbol\rho^{( \boldsymbol\mu, \boldsymbol\tau)}
		\Big]_{\substack{  \boldsymbol\mu \\ \boldsymbol\tau }}   
\Big[ \det \boldsymbol\sigma^{( \boldsymbol\tau, \boldsymbol\nu)}
		\Big]_{\substack{ \boldsymbol\tau \\ \boldsymbol\nu }} =
\Big[ \det \boldsymbol\sigma^{( \boldsymbol\mu, \boldsymbol\upsilon)}
		\Big]_{\substack{ \boldsymbol\mu \\ \boldsymbol\upsilon }}
\Big[ \det \boldsymbol J^{( \boldsymbol\upsilon, \boldsymbol\nu)}
		\Big]_{\substack{ \boldsymbol\upsilon \\ \boldsymbol\nu }},
\] 
where 
\begin{equation}			\label{Eq_Sec3.5Jordan}
\Big[ \det \boldsymbol J^{( \boldsymbol\upsilon, \boldsymbol\nu)}
		\Big]_{\substack{ 1\leqslant\upsilon_1<\cdots<\upsilon_l\leqslant m \\ 
						 			1\leqslant\nu_1<\cdots<\nu_l\leqslant m }}
\end{equation}
is an upper triangular 
matrix with $\lambda_{\boldsymbol\nu}$ on its diagonal, and 
\[
\Big[ \det \boldsymbol\sigma^{( \boldsymbol\tau, \boldsymbol\nu)}
		\Big]_{\substack{ 1\leqslant\tau_1<\cdots<\tau_l\leqslant m \\ 
						 			1\leqslant\nu_1<\cdots<\nu_l\leqslant m }} 
\]
is a non-singular matrix by the previous argument. Thus, the eigenvalues of 
\eqref{Eq_Sec3.5GrassmanMatrix}, are the same as the eigenvalues of \eqref{Eq_Sec3.5Jordan}, 
which plainly are the products $\lambda_{\boldsymbol\nu}$, and the theorem is proved.
\end{proof}
\begin{remark}
The matrix \eqref{Eq_Sec3.5GrassmanMatrix} is called the $l$-th {\it compound matrix}, or the 
$l$-th {\it adjugate}, of $\boldsymbol\rho$. A proof of the Sylvester-Franke theorem by 
induction, and several interesting historical notes with a rich bibliography can be found 
in \cite{Price}. Further proofs are in \cite{Tornheim} and \cite{Flanders}. The second 
part of our proof has intersection with \cite{Flanders} when $\boldsymbol\rho$ is 
diagonalizable, i.e. when $\boldsymbol J$ is diagonal. The natural environment of the 
compound matrices is the exterior algebra $\Lambda^l \C^m$.
\end{remark}
\begin{remark}
Rather interestingly, the $\boldsymbol{L\Delta U}$ factorization of the Hessian 
matrix is a cornerstone of the $\C^N$-saddle point method in \cite{PinnaViola}. 
\end{remark}
\begin{remark}
If we wish to find the eigenvectors of the compound matrix \eqref{Eq_Sec3.5GrassmanMatrix}, 
assuming that we already know the eigenvectors of $\boldsymbol\rho$, which are 
(some of) the columns of $\boldsymbol\sigma$ in the above proof, then we are confronted 
with the entirely combinatorial problem of finding the Jordan normal form of the 
compound matrix of $\boldsymbol J$, which is Jordan's normal form of $\boldsymbol\rho$. 
The solution of this problem is detailed in \cite{Aitken} and \cite{Littlewood}.
\end{remark}
\subsection{Asymptotic behavior of the minors}
Pituk \cite{Pituk2002} considered Poincar\'e-Perron type difference systems
\begin{equation}			\label{Eq_Sec3.6PitukSystem}	
\boldsymbol p_{n+1} = \left[\boldsymbol A + \boldsymbol B_n \right] \boldsymbol p_n, 
\end{equation}
where $\boldsymbol p_n\in\C^m$, the matrix $\boldsymbol A\in\C^{m\times m}$ is 
independent of $n$, and the sequence of matrices $\boldsymbol B_n\in\C^{m\times m}$ 
satisfies 
\[
\lim_{n\to\infty} \| \boldsymbol B_n \| = 0.
\]
Here, $\|\cdot\|$ can be any norm on $\C^{m\times m}$. 

Putting two theorems together, we have
\begin{theorem}{(Pituk \cite{Pituk2002} \cite{Pituk2011})} 
Let $\lambda_1,\dots,\lambda_m$ be the eigenvalues of $\boldsymbol A$. If $\boldsymbol p_n$ 
is a solution of \eqref{Eq_Sec3.6PitukSystem}, then 
either $\boldsymbol p_n=\boldsymbol 0$ for any sufficiently large $n$, or 
\begin{equation}			\label{Eq_Sec3.6PitukAsymptotic}
\lim_{n\to\infty} \sqrt[n]{\|\boldsymbol p_n \|}=|\lambda_j| \qquad 
											\text{ for some } j=1,\dots,m.
\end{equation}
Furthermore, if in \eqref{Eq_Sec3.6PitukAsymptotic} 
we have $\boldsymbol p_n \in\R_{\geqslant 0}^m$, then there exists 
an eigenvector $\boldsymbol q$ of $\boldsymbol A$ such that 
$\boldsymbol A \boldsymbol q = \lambda_j \boldsymbol q$ (with the same eigenvalue 
as in \eqref{Eq_Sec3.6PitukAsymptotic}) and $\boldsymbol q\in \R_{\geqslant 0}^m$.
\end{theorem}
Again, in \eqref{Eq_Sec3.6PitukAsymptotic} we can take any norm on $\C^m$. The limit 
equation we reported about in Sect. 3.2 above was obtained by choosing the $\ell_1$-norm 
in $\C^m$, $\boldsymbol A = \boldsymbol \Psi$ and $\boldsymbol B_n = \boldsymbol \Psi_n 
- \boldsymbol \Psi$, where $\boldsymbol \Psi_n$ and $\boldsymbol \Psi$ are the companion 
matrices in Sects. 3.2-3.3. 

The quoted theorem by Pituk require a very weak assumption on the sequence $\boldsymbol B_n$ 
and essentially no assumption on the matrix $\boldsymbol A$, which is very remarkable 
in comparison with previous results by Perron, M\'at\'e and Nevai, Coffman, Li, Trench, 
Pituk himself and other authors.

Combining \eqref{Eq_Sec3.6PitukAsymptotic} and \eqref{Eq_Sec3.5GrassmannDeterminant}, 
we get the following
\begin{theorem}
Let $x_n^{(1)},\dots,x_n^{(l)}$ be linearly independent solutions 
of \eqref{Eq_Sec3.1=0}. Suppose that \eqref{Eq_Sec3.1lim} holds, and 
let $\lambda_1,\dots,\lambda_m$ be the eigenvalues of the matrix $\boldsymbol \Psi$
in \eqref{Eq_Sec3.4companionMatrix}. Let 
\[
\boldsymbol x_n=\left[x_{n+i-1}^{(j)}\right]_{\substack{i=1,\dots,m \\ j=1,\dots,l}}.
\] 
Then
\[
\lim_{n\to\infty} \sqrt[n]{\Big\| \left[ \det\boldsymbol x_n^{(\boldsymbol\mu,\textbf{-})}
		\right]_{\substack{\boldsymbol\mu \\ \textbf{ - }}} 
		\Big\|}=|\lambda_{\nu_1}\cdots\lambda_{\nu_l}| 
				\qquad \text{ for some } 1\leqslant\nu_1<\cdots<\nu_l\leqslant m.
\]
\end{theorem}
\begin{proof}
By Sect. 3.1, our assumption on $x_n^{(1)},\dots,x_n^{(l)}$ imply that
\[
\det\left[ x_{n+i-1}^{(j)} \right]_{\substack{i=1,\dots,l \\ j=1,\dots,l}}\not=0 
\quad \text{ for infinitely many } n.
\]
We may apply \eqref{Eq_Sec3.6PitukAsymptotic} to the system 
\[
\left[ \det \boldsymbol x_{n+1}^{(\boldsymbol\mu,\textbf{-})}
		\right]_{\substack{\boldsymbol\mu \\ \textbf{-}}} = 
\left[ \det \boldsymbol \Psi_n^{(\boldsymbol\mu,\boldsymbol\nu)}
		\right]_{\substack{\boldsymbol\mu \\ \boldsymbol\nu}}
\left[ \det \boldsymbol x_n^{(\boldsymbol\mu,\textbf{-})}
		\right]_{\substack{\boldsymbol\mu \\ \textbf{-}}},
\]
because
\[
\left[ \det \boldsymbol \Psi_n^{(\boldsymbol\mu,\boldsymbol\nu)}
		\right]_{\substack{\boldsymbol\mu \\ \boldsymbol\nu}} = 
\left[ \det \boldsymbol \Psi^{(\boldsymbol\mu,\boldsymbol\nu)}
		\right]_{\substack{\boldsymbol\mu \\ \boldsymbol\nu}} +
\left[ \det \boldsymbol \Psi_n^{(\boldsymbol\mu,\boldsymbol\nu)} 
			- \det \boldsymbol \Psi^{(\boldsymbol\mu,\boldsymbol\nu)}
		\right]_{\substack{\boldsymbol\mu \\ \boldsymbol\nu}}
\]
and 
\[
\lim_{n\to\infty} \Big\| \left[ \det \boldsymbol \Psi_n^{(\boldsymbol\mu,\boldsymbol\nu)} 
			- \det \boldsymbol \Psi^{(\boldsymbol\mu,\boldsymbol\nu)}
		\right]_{\substack{\boldsymbol\mu \\ \boldsymbol\nu}}
\Big\| = 0.
\]
By the Sylvester-Franke theorem, the eigenvalues of 
\[
\left[ \det \boldsymbol \Psi^{(\boldsymbol\mu,\boldsymbol\nu)}
		\right]_{\substack{\boldsymbol\mu \\ \boldsymbol\nu}}
\]
are precisely the products $\lambda_{\nu_1}\cdots\lambda_{\nu_l}$, 
for $1\leqslant\nu_1<\cdots<\nu_l\leqslant m$. 
\end{proof}
\begin{remark}
The above result can be made more precise, using the Jordan normal form of the compound 
matrix of $\boldsymbol\Psi$, and we refer the reader to our previous remark 3.7.
\end{remark}
\section{Some applications of our criterion}
In this section we outline a concrete application of our criterion on two examples. The 
exposition that follows is a bit sketchy, for two reasons. The first one is that we want 
to keep the focus of the paper on the criterion itself, and the examples below are 
merely illustrative. The second reason is that we do not try here to optimize 
the parameters in the first example, see the end of subsection 4.1, nor we put special care 
in the general upper bound for the linear forms, see below. Thus, the experimental 
results we present here are very likely improvable; in the first example, with the help 
of the refined criterion, see Theorem 2.4 above, combined with the so-called permutation-group method; in the second example, with a clever estimate of the linear forms.

Let $\alpha_1,\dots,\alpha_m\in\C$ be distinct, and let $k,m,n\in\N$ with $k\geqslant1$. 
The $k$-th polylogarithm of $z$ is defined for $z\in\C$ with $|z|<1$, by
\[
{\rm Li}_k(z) = \sum_{l=1}^\infty \frac{z^l}{l^k}.
\]
We put, recursively,
\[
U_0(z) = \prod_{i=1}^m (z+\alpha_i)^{kn}, \quad 
U_j(z) = \frac{1}{n!} \frac{{\rm d}^n}{{\rm d} z^n} \big( z^n U_{j-1}(z) \big)\; (j=1,\dots,k).
\]
The polynomials $U_j(z)$ have degree $kmn$. Let $V_0(z)=z^{kmn} U_k(1/z)$. There 
exist $km$ polynomials $W_{i,j}(z)$ with $i=1,\dots,m$ and $j=1,\dots,k$ having degree 
not exceeding $kmn$ such that 
\begin{equation}			\label{Eq_Sec4PadePolysPoints}
V_0(z) {\rm Li}_j(-\alpha_i z) - W_{i,j}(z) = O(z^{kmn+n+1}) \qquad (z\to 0). 
\end{equation}
In other words, $V_0(z),W_{1,1}(z),\dots,W_{1,k},\dots,W_{m,1},\dots,W_{m,k}$ is a system of 
$(n,\dots,n)$ type II Pad\'e approximations 
to $1,{\rm Li}_1(-\alpha_1 z),\dots,{\rm Li}_k(-\alpha_1 z),\dots,
{\rm Li}_1(-\alpha_m z),\dots,{\rm Li}_k(-\alpha_m z)$ at $z=0$. This is a special 
case of a more general analytic construction introduced in \cite{DavidHirata-KohnoKawashima}. 
The case $k=1$ was introduced in \cite{Nikishin} and \cite{RhinToffin}, and the case $m=1$ 
was introduced, in the more general context of the Lerch functions, in 
\cite{Hata_Polylogarithms}. In particular, the above statement is a special 
case of \cite[Theorem 3.6]{DavidHirata-KohnoKawashima}. Moreover, the construction 
can be slighly twisted to obtain a $(km+1)\times (km+1)$ square matrix of polynomials 
whose determinant does not vanishes for any $n$, and is independent of $z$: 
see \cite[Proposition 5.1]{DavidHirata-KohnoKawashima}.

We shall need the following more explicit expressions for the polynomials $V_0$ and $W_{i,j}$:
\begin{align*}
V_0(z) & = \sum_{l_1 = 0}^{kn} \cdots \sum_{l_m = 0}^{kn} 
		\binom{n+l_1+\cdots+l_m}{n}^k \binom{kn}{l_1} \cdots \binom{kn}{l_m} 
		\alpha_1^{kn-l_1} \cdots \alpha_m^{kn-l_m} z^{kmn-l_1-\cdots-l_m} 	\\
W_{i,j}(z) & = \sum_{l_1 = 0}^{kn} \cdots \sum_{l_m = 0}^{kn} 
		\binom{n+l_1+\cdots+l_m}{n}^k \binom{kn}{l_1} \cdots \binom{kn}{l_m} 
		\alpha_1^{kn-l_1} \cdots \alpha_m^{kn-l_m}  \\
	&   \times
	\sum_{s=1}^{l_1+\cdots+l_m} \frac{(-1)^{s-1}}{s^j} \alpha_i^s z^{s+kmn-l_1-\cdots-l_m} 
	\qquad (i=1,\dots,m;j=1,\dots,k).
\end{align*}
Taking into account \eqref{Eq_Sec4PadePolysPoints}, if $|\alpha_i|<1$ for $i=1,\dots,m$, then
\[
|V_0(1) {\rm Li}_j(-\alpha_i) - W_{i,j}(1)| \leqslant 
{\rm Li}_j(|\alpha_i|) |\alpha_i|^{(km+1)n}
\max_{l=0,\dots,kn} \binom{n+ml}{n}^k \binom{kn}{l}^m.
\]

For a generalization of the main result of \cite{Marcovecchio2006} to values of the 
polylogarithm at algebraic points outside the unit disc, we refer the reader to the 
recent paper \cite{FischlerRivoal_Gfunctions}. That requires much deeper arguments, which we do 
not tackle here.

From
\[
\binom{n+ml}{n}^k \binom{kn}{l}^m \leqslant \binom{n+kmn}{n}^k \binom{kn}{[\frac{kn}{2}]}^m
\]
we infer that
\begin{multline*}
\log |V_0(1) {\rm Li}_j(-\alpha_i) - W_{i,j}(1)| \\ =
(km+1)n\log |\alpha_i| +\big((km+1)\log(km+1) - km \log(km) + m \log 2)kn + O(\log n),
\end{multline*}
for $n\to\infty$. This will be used in the second example.
\subsection{First example}
In the first example we focus on the linear independence 
of 
\[
1,{\rm Li}_1(1/q),{\rm Li}_2(1/q),{\rm Li}_1(2/q),{\rm Li}_2(2/q)
\] 
over $\Q$ for all sufficiently large \emph{positive} integers $q$. Let $\alpha=-1/q$ and $\beta=-2/q$, and let
\[
u(z;n) = \frac{1}{n!} \frac{{\rm d}^n}{{\rm d} z^n} \left( 
\frac{z^n}{n!} \frac{{\rm d}^n}{{\rm d} z^n} \big( z^n (z+\alpha)^{2n} (z+\beta)^{2n}\big) \right),
\]
\[
v_1(z;n) = \alpha \int\limits_0^1 \frac{u(z;n)-u(-\alpha t;n)}{z+\alpha t}\, {\rm d} t, \quad
v_2(z;n) = \alpha \int\limits_0^1 \frac{u(z;n)-u(-\alpha t;n)}{z+\alpha t}\,
																		\log t\, {\rm d} t, 
\]
\[
w_1(z;n) = \beta \int\limits_0^1 \frac{u(z;n)-u(-\beta t;n)}{z+\beta t}\, {\rm d} t, \quad
w_2(z;n) = \beta \int\limits_0^1 \frac{u(z;n)-u(-\beta t;n)}{z+\beta t}\,
																		\log t\, {\rm d} t; 
\]
see \cite[p.285]{RhinToffin} and \cite[p.375]{Hata_Dilogarithm} for similar formulas. 
Roughly, when $q$ is large the linear forms 
\begin{equation}			\label{Eq_Sec4.1linearforms}
\begin{split}
u(1;n){\rm Li}_1(1/q)-v_1(1;n),\qquad & u(1;n){\rm Li}_2(1/q)-v_2(1;n), 		\\
u(1;n){\rm Li}_1(2/q)-w_1(1;n),\qquad & u(1;n){\rm Li}_2(2/q)-w_2(1;n)
\end{split}
\end{equation}
are small. However, the coefficients $u(1;n),v_1(1;n), v_2(1;n), w_1(1;n), w_2(1;n)$ 
are rational. It is sufficient to multiply each of them by 
\[
q^{4n} d_{4n}^2,
\]
where $d_n$ is the least common multiple of the integers $1,\dots,n$, to obtain 
approximations with integer coefficients. More precisely, our refined criterion, i.e. 
Theorem 2.4 above, is less demanding, but in this first attempt we prefer to skip 
these nuances. Computations show that using just the same sequences 
of approximations we experimentally find out that the five numbers 
$1,{\rm Li}_1(1/q),{\rm Li}_2(1/q),{\rm Li}_1(2/q),{\rm Li}_2(2/q)$ are linearly independent 
over $\Q$ for all $q\geqslant 1323$, and that four out of the same five numbers are linearly 
independent over $\Q$ for all $q\geqslant 1287$. The improvement on the range for $q$ in the 
second case depends on the asymptotic behavior of the $2\times 2$ determinants of linear forms, 
as showed in Theorem 3.3 above. For comparison, we recall 
that $1,{\rm Li}_1(1/q),{\rm Li}_2(1/q)$ are known to be linear independent over $\Q$ 
for $q\geqslant 6$, and $1,{\rm Li}_1(2/q),{\rm Li}_2(2/q)$ are linear independent 
over $\Q$ for $q\geqslant 51$, see \cite[p.94]{RhinViola2019}. 

To achieve our plan rigorously, one should compute a linear recurrence relation satisfied 
by all the coefficients $u(1;n),v_1(1;n), v_2(1;n), w_1(1;n), w_2(1;n)$. According to the 
experimental style of this section, we proceed differently. The coefficient $u(1;n)$ can 
be easily written as a double Cauchy integral, so that 
\begin{equation}			\label{Eq_Sec4.1limlogun}
\lim_{n\to\infty} \frac{1}{n} \log |u(1;n)| 
\end{equation}
depends on the critical values of the function
\[
f(s,t)=\frac{s(s-1/q)^2 (s-2/q)^2 t}{(s-t)(t-1)}. 
\]
Solving 
\[
\frac{\partial }{\partial t} f(s,t) = 0,
\]
we obtain $s=t^2$. If we were using the $\C^2$ saddle point method \cite{Hata_Saddle}, 
we would look for the critical values of $f(t^2,t)=\big(g(t)\big)^2$, where
\[
g(t) = \frac{t(t^2-1/q)(t^2-2/q)}{t-1}.
\] 
If we solve $g^\prime(t)=0$, i.e. we find the roots $t_1,\dots,t_5$ of 
\[
4t^5-5 t^4 -\frac{6}{q} t^3 + \frac{9}{q} t^2 -\frac{2}{q^2} = 0,
\]
and compute 
\[
\log |t_i| + \log |t_i^2 - 1/q| + \log |t_i^2 - 2/q| - \log |t_i-1| + 2 \log q + 4, 
\]
for $i=1,\dots,5$, the maximum of those values, doubled, is an 
upper bound for \eqref{Eq_Sec4.1limlogun}, and the second maximum, 
doubled, is an upper bound for the linear forms
\[
\begin{split}
q^{4n}d_{4n}^2 u(1;n){\rm Li}_1(1/q)-q^{4n}d_{4n}^2 v_1(1;n),\qquad & 
q^{4n}d_{4n}^2 u(1;n){\rm Li}_2(1/q)-q^{4n}d_{4n}^2 v_2(2;n), 				\\
q^{4n}d_{4n}^2 u(1;n){\rm Li}_1(2/q)-q^{4n}d_{4n}^2 w_1(1;n),\qquad & 
q^{4n}d_{4n}^2 u(1;n){\rm Li}_2(2/q)-q^{4n}d_{4n}^2 w_2(1;n). 
\end{split}
\] 
Suppose that the roots $t_1,\dots,t_5$ are ordered in such a way that 
\[
|g(t_1)|>|g(t_2)|>|g(t_3)|>|g(t_4)|>|g(t_5)|. 
\]
Note that $|g(t_2)|=O(|q|^{-\frac{5}{2}})$ for $q\to\infty$ 
by \eqref{Eq_Sec4PadePolysPoints}. The five numbers 
\[
1,{\rm Li}_1(1/q),{\rm Li}_2(1/q),{\rm Li}_1(2/q),{\rm Li}_2(2/q)
\]
are linearly independent over $\Q$ if 
\[
\log |g(t_2)| + 2 \log q + 4 < 0,
\]
and at least four among the above five numbers are linearly independent over $\Q$ if 
\[
\log |g(t_2)| + \log|g(t_3)| + 4 \log q + 8 < 0. 
\]
This explains the difference in the ranges for $q$ in the two cases. 

It is not very difficult to turn the above heuristic argument into a fully rigorous one. 
Let us outline how to do this. First of all, we need a characterization of the polynomials 
$u(z;n)$ in terms of certain orthogonality conditions. This means that
\[
\int\limits_0^1 t^l u(-\gamma t;n) (\log t)^j \, {\rm d} t = 0 \quad 
\gamma=\alpha,\beta;\; j=0,1;\; l=0,\dots,n-1,
\]
and that any polynomial $U(z)$ in $z$ of degree not exceeding $4n-1$ and satisfying 
\[
\int\limits_0^1 t^l U(-\gamma t) (\log t)^j \, {\rm d} t = 0 \quad 
\gamma=\alpha,\beta;\; j=0,1;\; l=0,\dots,n-1,
\]
must be identically zero. Secondly, we can find five  polynomials $A_0(z;n),\dots,A_4(z;n)$ 
in $z$ and $n$, not all zero, such that 
\[
\deg_z A_0(z;n)\leqslant 7,\; \deg_z A_1(z;n)\leqslant 11,\; \deg_z A_2(z;n)\leqslant 10,\; 
\deg_z A_3(z;n)\leqslant 9,\; \deg_z A_4(z;n)\leqslant 8,\; 
\]
and that 
\[
\deg_z \left(\sum_{j=0}^4 A_j(z;n) u(z;n-j) \right) < 4(n-12).
\]
By the orthogonality conditions above, we get
\[
\sum_{j=0}^4 A_j(z;n) u(z;n-j) = 0.
\]
After dividing by a suitable power of $n$, this is a Poincar\'e-Perron-Pituk-type recurrence, 
and we may apply our results in Sects. 2 and 3. Using the orthogonality conditions again, it 
is easy to see that the polynomials $v_1(z;n),v_2(z;n),w_1(z;n)$ and $w_2(z;n)$ satisfy the 
same recurrence relation as $u(z;n)$. Moreover, instead of actually computing the recurrence, 
maybe with the help of the algorithm in \cite{WilfZeilberger} implemented in some computer 
algebra system, one can also use the explicit form 
\[
u(z;n)= \sum_{p=0}^{2n} \sum_{q=0}^{2n} \binom{2n}{p} \binom{2n}{q} \binom{5n-p-q}{n}^2 
									\alpha^p \beta^q z^{4n-p-q}
\]
to obtain explicitly the {\it limit} equation, i.e. \eqref{Eq_Sec3.4polynomial}, 
of the recurrence, like, e.g., in \cite[Theorem 5.1]{Marcovecchio2014}. We understand 
that this is just a sketch, but that was all what we promised. 

We do not consider the sets of numbers $1,{\rm Li}_1(\alpha),{\rm Li}_2(\alpha),
{\rm Li}_1(\beta),{\rm Li}_2(\beta)$ when $(\alpha,\beta)$ is one of 
\[
(1/q,2/q), \quad (1/q,-2/q), \quad (-1/q,2/q),
\] 
because we experimentally found that in each of these cases $g(t_2)$ and $g(t_3)$ happen 
to be complex conjugate solutions of a polynomial of degree $5$, 
therefore $|g(t_2)|=|g(t_3)|$, so that our criterion would not have an interesting 
application. On the other hand, a way to circumvent this difficulty could be the use 
of approximations more general than those considered above, obtained, e.g. changing 
$u(z;n)$ into 
\[
\frac{z^{-q_2 n}}{((p_2-q_2) n)!} \frac{{\rm d}^{(p_2-q_2)n}}{{\rm d} z^{(p_2-q_2)n}} \; 
\frac{z^{(p_2-q_1)}}{((p_1-q_1)n)!} \frac{{\rm d}^{(p_1-q_1)n}}{{\rm d} z^{(p_1-q_1)n}} 
\big( z^{p_1 n} (z+\alpha)^{rn} (z+\beta)^{sn}\big).
\]
This will be the subject of some future paper. We also remark that there is another strategy 
that is totally independent of consideration of linear recurrence sequences. The linear 
forms \eqref{Eq_Sec4.1linearforms} can be written as double integrals, as to the two on 
the left of \eqref{Eq_Sec4.1linearforms}, or sums of double integrals, as to the two on 
the right of \eqref{Eq_Sec4.1linearforms}, by using the orthogonality properties of $u(z;n)$ 
like in the papers \cite{RhinToffin} and \cite{Hata_Dilogarithm}, and then applying 
the $\C^2$--saddle method. As to the upper bound, in the $2\times 2$ determinant one 
can get ride of one of the four double integrals involved, thus carrying the 
value $|g(t_3)|$ into play. For the non-vanishing assumption, it could be handled 
through a lower bound of the determinant, again using the $\C^2$--saddle method.
\subsection{Second example}
This example aims to illustrate the refined criterion, i.e. Theorem 2.4. To keep the 
exposition as simple as possible, we now disregard the contribution that comes from the 
application of Theorem 3.3, and just use the trivial estimation of the determinants, i.e 
Hadamard's inequality. Let us consider the $km+1$ numbers
\begin{equation}			\label{Eq_Sec4.2km+1Lik}
1,\{{\rm Li}_1(1/lq),\dots,{\rm Li}_k(1/lq) \; |\; l=1,\dots,m\},
\end{equation}
where $q$ is a sufficiently large integer, positive or negative. With the 
notation at the beginning of the section, we set $\alpha_l=1/lq$. A special case of \cite[Theorem 2.1]{DavidHirata-KohnoKawashima} is the following: the $1+km$ numbers 
above are linearly independent over $\Q$ whenever
\begin{equation}			\label{eq_Sec4.2dimensionDHK}
\log |q|>km\big(k + \log d_m + k\log (5/2) \big) + k\log 3.
\end{equation}
We have 
\[
d_m^{kn} U(1)\in \Z[1/q],\qquad d_m^{kmn} d_{kmn}^j W_{i,j}(1)\in\Z[1/q] \quad
															(j=1,\dots,k;i=1,\dots,m),
\] 
and
\[
\lim_{n\to\infty} \frac{1}{n} \log (d_m^{kmn} d_{kmn}^k) = km\log d_m + k^2 m.
\]
Combining
\[
q^{kmn} d_m^{kn} U(1)\in \Z,\qquad q^{kmn} d_m^{kmn} d_{kmn}^j W_{i,j}(1)\in\Z \quad
															(j=1,\dots,k;i=1,\dots,m)
\] 
with the above lower bounds for the linear forms, and using the non-vanishing result 
in \cite[Proposition 5.1]{DavidHirata-KohnoKawashima}, we see that the $km+1$ numbers 
\eqref{Eq_Sec4.2km+1Lik} are linear independent over $\Q$ if
\[
\log|q|> k^2 m +km \log (2d_m) +k (km+1)\log(km+1)-k^2 m\log(km),
\]
and, {\it a fortiori}, if
\[
\log|q|>k^2 m +km \log (2d_m) +k\log(km+1) +k,
\]
which improves upon \eqref{eq_Sec4.2dimensionDHK}. Numerically, for $k=m=10$ we 
require $\log|q|\geqslant 1909$, and for $k=m=11$ we need $\log |q|\geqslant 2717$; 
compare with \cite[Example 6.1]{DavidHirata-KohnoKawashima}.

Using our Theorem 2.4, we can prove a lower bound for the dimension of the vector space 
spanned over $\Q$ by the above numbers, say $\delta(k,m)$, for a wider range for $q$.
Looking more closely at the definition of the polynomials $W_{i,j}(z)$, we have
\[
d_m^{(k-1)mn} d_{kmn}^j W_{i,j}(1)\in\Z[1/q]\; \text{ if } j\leqslant k-1, \quad 
d_{m-1}^{k(m-1)n} d_m^{kn} d_{kmn}^j W_{i,j}(1)\in\Z[1/q]\; \text{ if } i\leqslant m-1.
\]
Thus
\[
d_m^{(k-1)mn} d_{m-1}^{mn} d_{kmn}^k W_{i,j}(1)\in\Z[1/q]\; \text{ if } (i,j)\not=(m,k).
\]
Similarly, we have:
\begin{itemize}

\item $\delta(k,m)\geqslant km$ if $k,m\geqslant 2$ and
\[
2\log|q|>2k^2 m +(2k-1)m \log d_m +m\log d_{m-1} +2km\log 2 +2k\log(km+1) +2k;
\]

\item $\delta(k,m)\geqslant km-1$ if $k,m\geqslant 3$ and
\[
3\log|q|>3k^2 m +(3k-2)m \log d_m +2m\log d_{m-1} +3km\log 2 +3k\log(km+1) +3k.
\]
\end{itemize}
Now let $k,m\geqslant 4$. Refining again the above denominator estimate, we remark that
\[
d_m^{(k-2)mn} d_{m-1}^{mn} d_{m-2}^{mn} d_{kmn}^k W_{i,j}(1)\in\Z[1/q] 
		\; \text{ if } (i,j)\not=(m,k),(m-1,k),(m,k-1).
\]
Therefore 
\begin{itemize}

\item
$\delta(k,m)\geqslant km-2$ if $k,m\geqslant 4$ and
\[
4\log|q|>4k^2 m +(4k-3)m \log d_m +2m\log d_{m-1} +m\log d_{m-2} +4km\log 2 +4k\log(km+1) +4k;
\]

\item $\delta(k,m)\geqslant km-3$ if $k,m\geqslant 5$ and
\[
5\log|q|>5k^2 m +(5k-4)m \log d_m +2m\log d_{m-1} +2m\log d_{m-2} +5km\log 2 +5k\log(km+1) +5k;
\]

\item $\delta(k,m)\geqslant km-4$ if $k,m\geqslant 6$ and
\[
6\log|q|>6k^2 m +(6k-5)m \log d_m +2m\log d_{m-1} +3m\log d_{m-2} +6km\log 2 +6k\log(km+1) +6k.
\]
\end{itemize}
The argument can be continued as far as $k$ and $m$ are large enough, and each time 
we can cut a larger triangular corner from the range for $i$ and $j$, and get a better, 
i.e. smaller, denominator outside that corner.

Numerically, $\delta(11,11)\geqslant 1+121-(1+2+3+4)=112$ if
\[
10 \log |q|>10 11^3 + (110-9)10\log(d_{11}) + 55 \log(d_{10}) + 44(\log d_8) + 1210 \log 2 
+ 110 \log(122) + 110, 
\]
i.e. $|q|>e^{2585}$.

\begin{remark}
Further examples can be obtained, e.g., with the set of numbers
\[
1,\{{\rm Li}_1(1/p^{l-1}q),\dots,{\rm Li}_k(1/p^{l-1}q) \; |\; l=1,\dots,m\},
\]
where $p$ is a positive integer.
\end{remark}
\begin{remark}
It would be interesting to obtain an application of our criterion, or of a 
suitable quantitative version of it, to the linear independence of values 
of $G$-functions at several points, in the spirit of \cite[Theorem 1]{Chudnovskys}, 
whose proof uses Hermite-Pad\'e approximations of type II. 
\end{remark}

\section*{Acknowledgments}
The core of the present note, a short less-than-seven-pages draft without applications, 
was written several years ago partly during, partly after a stay at the Centre International 
de Rencontres Math\'ematiques de Luminy, France. However, my attention to the topic was 
recently refreshed by the papers \cite{DavidHirata-KohnoKawashima} and \cite{ViolaZudilin}, 
and specially by the proof of \cite[Lemma 5.3]{ViolaZudilin}. Along the time, I 
had the pleasure of chatting on this topic with F.Amoroso, M.Laurent and W.Zudilin;
a special thank to them, and to whom else made all this possible, in a way or another. 

Some computations in Sect. 4 were made with the help of the free software Pari/GP \cite{PARI}.

\end{document}